\documentclass[11pt]{amsart}

\usepackage{verbatim}
\usepackage{amssymb}
\usepackage{amscd}
\usepackage{amsmath}
\usepackage{setspace}
\usepackage{fullpage}

\newcommand{\mat}[4]{{\setlength{\arraycolsep}{0.5mm}\left[
\begin{array}{cc}#1&#2\\#3&#4\end{array}\right]}}
\newtheorem{theorem}{Theorem}[subsection]

\newtheorem{corollary}[theorem]{Corollary}
\newtheorem{lemma}[theorem]{Lemma}
\newtheorem{proposition}[theorem]{Proposition}
\numberwithin{equation}{subsection}

\newtheorem{case}{Case}

\newcommand{\Q}{\mathbb Q}
\newcommand{\Ht}{\widetilde{\mathbb H}}
\renewcommand{\H}{\mathbb H}
\newcommand{\G}{\widetilde{G}}
\newcommand{\Z}{\mathbb Z}
\newcommand{\C}{\mathbb C}
\newcommand{\R}{\mathbb R}
\newcommand{\F}{\mathbb F}
\newcommand{\A}{\mathbb A}
\newcommand{\bs}{\backslash}
\renewcommand{\i}{\iota}

\renewcommand{\Re}{\mathrm{Re}}

\newcommand{\tH}{\widetilde H}
\newcommand{\tR}{\widetilde R}
\newcommand{\tF}{\widetilde F}

\newcommand{\tl}{\widetilde l}

\title[Pullbacks of {E}isenstein series and critical $L$-values] {Pullbacks of {E}isenstein series from $GU(3,3)$ \\and\\ critical $L$-values for
      $GSp(4) \times GL(2)$}
\author{Abhishek Saha}
\date{\today}
 \address{Department of Mathematics 253-37 \\ California Institute of Technology \\ Pasadena, California 91125 \\ USA}
\email{saha@caltech.edu}

\begin{document}
\bibliographystyle{plain}

\begin{abstract}Let $F$ be a genus two Siegel newform and $g$ a classical newform, both of squarefree levels and of equal weight $\ell$. We prove a pullback formula for certain Eisenstein series --- thus generalizing a construction of Shimura --- and use this to derive an explicit integral representation for the degree eight $L$-function $L(s, F \times g)$. This integral representation involves the pullback of a simple Siegel-type Eisenstein series on the unitary group $GU(3,3)$. As an application, we prove a reciprocity law --- predicted by Deligne's conjecture --- for the critical special values $L(m, F \times g)$ where $m \in \Z, 2 \le m \le \frac{\ell}{2}-1$.
\end{abstract}

\maketitle
\section*{Introduction}

If $L(s, \mathcal{M})$ is an arithmetically defined (or motivic) $L$-series associated to an arithmetic object $\mathcal{M}$, it is of interest to study its values at certain critical points $s=m$. For these critical points, conjectures due to Deligne predict that the corresponding $L$-values satisfy the following reciprocity law:
\begin{enumerate}
\item $L(m,\mathcal{M})$ is the product of a suitable transcendental number $\Omega$ and an algebraic number $A(m,\mathcal{M})$.

\item If $\sigma$ is an automorphism of $\C$, then $A(m,\mathcal{M})^\sigma = A(m, \mathcal{M}^\sigma)$.
\end{enumerate}
In this paper, we prove a key special case of the above conjecture when $\mathcal{M}$ corresponds to the product $F \times g$ where $F$ is a Siegel modular form and $g$ a classical modular form. Precisely, fix odd, squarefree integers $M,N$. Let $F$ be a genus two Siegel newform of level $M$ and $ g$ an elliptic newform of level $N$; see Section~\ref{s:intrep} for the definitions of these terms. We assume that $F$ and $ g$ have the same even integral weight $\ell$ and have trivial central characters. We also make the following assumption about $F$:

Suppose $$F(Z) = \sum_{S>0} a(S)e(\text{tr}(SZ)) $$ is the Fourier expansion; then we assume that

\begin{equation}\label{fundamentalrestriction0}a(T) \neq 0
\text{ for some } T
= \begin{pmatrix} a& \frac{b}{2} \\ \frac{b}{2} & c \end{pmatrix}\end{equation} such that $-d = b^2
-4ac$  is the discriminant of the imaginary quadratic field $L = \Q(\sqrt{-d}),$
  and all primes dividing $MN$ are inert in  $L$.

One can associate a degree eight $L$-function $L(s, F\times g)$ to the pair $(F, g)$. We prove a reciprocity law (see Theorem~\ref{t:specialreciprocity} below) for the critical points $\{m: 2\le m \le \frac{\ell}{2}-1, m\in \Z\}$ of this $L$-function. As is often the case for such problems, the key ingredient in our proof is the interpretation of the transcendental factor as the period arising from a certain integral representation. In Section~\ref{section1} we associate to a Hecke character $\Lambda$ of $L$ a Siegel Eisenstein series $E_{\Upsilon}(g,s)$ on $GU(3,3;L)(\A)$. Let $R$ denote the subgroup of elements $(h_1,h_2) \in GSp(4) \times GU(1,1;L)$ for which $h_1, h_2$ have the same multiplier. We define in Subsection~\ref{s:embeddingdef} an embedding $\iota :R\hookrightarrow GU(3,3;L)$. Let $\Phi$, $\Psi$ denote the adelizations of $F, g$ respectively. We can extend the definition of $\Psi$ to $GU(1,1;L)(\A)$ by defining $\Psi(ag) = \Psi(g)$ for all $a \in L^\times(\A), g\in GL(2)(\A)$. Our integral representation is as follows.\bigskip

\textbf{Theorem~\ref{t:secondintegral}.}
\emph{We have $$\int_{r \in Z(\A) R(\Q) \bs R(\A)} E_{\Upsilon}( \i(g_1,g_2) ,s)\overline{\Phi}(g_1)\Psi(g_2) \Lambda^{-1}(\det g_2)dg = A(s)L(3s + \frac{1}{2}, F \times g)$$ where $r = (g_1, g_2)$, $\Lambda$ is a suitable Hecke character of $L$ and $A(s)$ is an explicit normalizing factor, defined in Section~\ref{s:intrep}.}\bigskip

The first step towards proving Theorem~\ref{t:secondintegral} was achieved in our earlier work~\cite{lfshort} where we extended an integral representation due to Furusawa. That integral representation involved a complicated \emph{Klingen} Eisenstein series attached to the cusp form $g$. The technical heart of this paper is a certain pullback formula (Theorem~\ref{t:pullback}) that expresses our earlier Eisenstein series as the inner product of the cusp form and the pullback of the simpler higher-rank \emph{Siegel} Eisenstein series $E_{\Upsilon}$. Formulas in this spirit were first proved in a classical setting by Shimura~\cite{shibook1}. Unfortunately, Shimura only considers certain special types of Eisenstein series in his work which does not include ours (except in the full level case $M=1, N=1$). Furthermore his methods are \emph{classical} and cannot be easily modified to deal with our case. The complicated sections at the ramified places and the need for precise factors make the \emph{adelic} language the right choice for our purposes. We provide a complete proof of the pullback formula for our Eisenstein series which explicitly gives the precise factors at the ramified places needed by us.

Combining the pullback formula with our previous work, we deduce Theorem~\ref{t:secondintegral}. It seems appropriate to mention here that the referee of our paper~\cite{lfshort} has indicated it may have been well known to some experts that one could use such a pullback formula to rewrite the Furusawa integral representation.

 From Theorem~\ref{t:secondintegral}, we easily conclude that $L(s, F\times g)$ is a meromorphic function whose only possible pole on the right of the critical line $\text{Re}(s) = \frac{1}{2}$ is simple and at $s= 1$. Moreover, with the aid of rationality results due to Garrett and Harris and the theory of nearly holomorphic functions due to Shimura, we prove the following Theorem.\bigskip

\textbf{Theorem~\ref{t:specialreciprocity}.} \emph{Suppose that the Fourier coefficients of $F$ and $g$ are totally real and algebraic and that $\ell \ge 6$. For a positive integer $k$, $1\le k \le \frac{\ell}{2}-2$, define $$A(F,g;k) = \frac{L(\frac{\ell}{2}-k, F \times g)}{\pi^{5\ell-4k-4}\langle F, F\rangle \langle  g,  g \rangle}.$$ Then we have, \begin{enumerate} \item $A(F,g;k)$ is algebraic \item For an automorphism $\sigma$ of $\C$, $A(F,g;k)^\sigma = A(F^\sigma, g^\sigma;k).$ \end{enumerate}}\bigskip

We remark here that the completely unramified case $M=1, N=1$ of the above theorem was already known by the works of Heim~\cite{heim} and B\"ocherer--Heim~\cite{heimboch}, who used a very different integral representation from the one in this paper. Also, just the algebraicity part of the above Theorem (i.e. part (a)) has been proved for the right-most critical value (corresponding to $k=1$) in various settings earlier by Furusawa~\cite{fur}, Pitale--Schmidt~\cite{pitsch} and the author~\cite{lfshort}.

To relate Theorem~\ref{t:specialreciprocity} to the conjecture of Deligne for motivic $L$-functions mentioned at the beginning of this introduction, we note that Yoshida~\cite{yosh} has shown that the set of all critical points for $L(s, F\times g)$ is $\{m: 2-\frac{\ell}{2} \le m \le \frac{\ell}{2} - 1, m\in \Z\}$. In particular, the critical points are always non-central (since the weight $\ell$ is even) and so the $L$-value is expected to be non-zero. Assuming the existence of a motive attached to $F$ (this seems to be now known for our cases by the work of Weissauer~\cite{weiss}) and the truth of Deligne's conjecture for the standard degree 5 $L$-function of $F$, Yoshida also computes the corresponding motivic periods. According to his calculations, the relevant period for the point $m$ is precisely the quantity $\pi^{4m+3\ell -4}\langle F, F\rangle \langle g, g \rangle$ that appears in our theorem above (once we substitute  $m= \frac{\ell}{2}-k$). We note here that Yoshida only deals with the full level case; however, as the periods remain the same (up to a rational number) for higher level, his results remain applicable to our case.

Thus, Theorem~\ref{t:specialreciprocity} is compatible with (and implied by) Deligne's conjecture, and furthermore, it covers all the critical values to the \emph{right} of Re($s)=\frac{1}{2}$ \emph{except} for the $L$-value at the point $1$. The proof for the critical values to the left of Re($s)=\frac{1}{2}$ would follow from the expected functional equation. Extending our result to $L(1, F\times g)$ is intimately connected to proving the analyticity of the $L$-function at that point (see Corollary~\ref{coranal}). These questions, related to analyticity and the functional equation are also of interest for other applications and will be considered in a future paper. In particular, once analyticity results are known for \emph{all} $GL(1)$ and $GL(2)$ twists of $F$, one could try using the converse theorem to lift $F$ to $GL(4)$. This is currently work in progress with A. Pitale and R. Schmidt.

We also note that the integral representation (Theorem~\ref{t:secondintegral}) is of interest for several other applications. For instance, we hope that this integral representation will pave the way
to certain new results involving stability, hybrid subconvexity, and non-vanishing results for the $L$-function under consideration following the methods of~\cite{micram}. We are also hopeful that we can prove results related to non-negativity of the central value $L(\frac{1}{2}, F)$. These results appear to be new for holomorphic Siegel modular forms. For example, the non-negativity result is known in the case of \emph{generic} automorphic representations by Lapid and Rallis~\cite{lapral}; however, automorphic representations associated to Siegel modular forms are never generic. Another interesting application of the integral representation would be to the construction of $p$-adic $L$-functions.

We expect most of the results of this thesis to hold for arbitrary totally real base fields. It would be particularly interesting to work out the special value results when the Hilbert-Siegel modular forms have different weights for each Archimedean place. This case will be considered in a future work.

We briefly summarize the logical structure of this paper. Section~\ref{section1} lays down the basic definitions concerning the Eisenstein series that will be used throughout the paper. In Section~\ref{s:pullbackstatement}, we state the crucial pullback formula (Theorem~\ref{t:pullback}). Roughly speaking, the pullback formula says that for a suitable choice of section $\Upsilon$, the Petersson inner product $\langle E_{\Upsilon}( \i(g,h) ,s), \Psi(h)\rangle$ essentially equals a particular Klingen Eisenstein series $E_{\Psi,\Lambda}(g, s)$ living on $GU(2,2)$. The proof of the pullback formula involves extensive local harmonic analysis as well as a careful choice of local sections. Sections~\ref{s:section3} and ~\ref{s:section4} are devoted to these computations and are possibly of independent interest. These local results are used in Section~\ref{s:pullback} to prove the pullback formula. In Section~\ref{s:intrep} we derive the crucial integral representation (Theorem~\ref{t:secondintegral}) for $L(s, F \times g)$ by combining the pullback formula with a result from~\cite{lfshort} that says that $\langle E_{\Psi,\Lambda}(g, s), \Phi(g, s)\rangle$ essentially equals $L(3s+\frac{1}{2}, F \times g)$. We rewrite our integral representation classically in Theorem~\ref{t:classicalintegral}. In Section~\ref{s:algebraicity}, we recall various rationality results relating to Petersson inner products, Eisenstein series and nearly holomorphic modular forms. These results are due to Garrett, Harris and Shimura and are the key tools that when applied on our integral representation lead to the proof, in Section~\ref{s:mainresult}, of our main result (Theorem~\ref{t:specialreciprocity}).

\subsection*{Acknowledgements}The author thankfully acknowledges his use of the software MAPLE for performing many of the computations for this paper.

The author thanks W.T. Gan and M. Harris for their valuable suggestions and P. Nelson for carefully reading through a draft of this paper. Thanks are also due to the referee for suggestions which significantly improved this presentation.

This work was done while the author was a graduate student at Caltech and represents part of his Ph.D. dissertation. The author thanks his advisor Dinakar Ramakrishnan for guidance, support and many helpful discussions.

\subsection*{Notation}
The symbols $\Z$, $\Z_{\ge0}$, $\Q$, $\R$, $\C$, $\Z_p$ and $\Q_p$ have the usual meanings. $\A$ denotes the ring of adeles of $\Q$, $\A_f$ the finite adeles. For a complex number $z$, $e(z)$ denotes $e^{2\pi i z}$.

For a matrix $M$ we denote its transpose by $M^t$. Denote by $J_n$ the $2n$ by $2n$ matrix given by $$J_n =
\begin{pmatrix}
0 & I_n\\
-I_n & 0\\
\end{pmatrix}.$$\\ We use $J$ to denote $J_2.$

For a positive integer $n$ define the group $GSp(2n)$ by $$GSp(2n,R) = \{g \in GL_{2n}(R) | g^tJ_ng = \mu_n(g)J_n \text{ for some }
\mu_n(g) \in R^\times\}$$ for any commutative ring $R$.

Define $Sp(2n)$ to be the subgroup of $GSp(2n)$ consisting of elements $g_1\in GSp(2n)$ with $\mu_n(g_1)=1$.

For an imaginary quadratic extension $L$ of $\Q$ define $$ GU(n,n)= GU(n,n;L)$$ by
$$GU(n,n)(\Q) =  \{g \in GL_{2n}(L) | (\overline{g})^tJ_ng = \mu_n(g)J_n, \mu_n(g)
\in \Q^\times\}$$ where $\overline{g}$ denotes the conjugate of $g$.

Let $\tH=GU(3,3), \tH_1=U(3,3), H=GSp(6), H_1 = Sp(6), \G = GU(2,2),$ $\G_1 = U(2,2)$, $G=GSp(4), G_1 = Sp(4), \tF = GU(1,1), \tF_1=U(1,1)$.

Define \begin{align*}\Ht_n&= \{ Z \in M_{n}(\C) | i( \overline{Z} - Z) \text{ is positive definite}\},\\
    \H_n&= \{ Z \in M_n(\C) | Z =Z^t,i( \overline{Z} - Z) \text{ is positive definite}\}.\end{align*}
For $g=
\begin{pmatrix} A&B\\ C&D \end{pmatrix} \in GU(n,n)(\R)$, $Z\in \Ht_n$ define $$J(g,Z) = CZ + D.$$ The same definition works for $g\in GSp(2n)(\R), Z\in \H_n.$

For a commutative ring $R$ we denote by $I(2n, R)$ the Borel subgroup of $GSp(2n,R)$ consisting of the set of matrices that look like $\begin{pmatrix}
A & B\\
0 & \lambda (A^t)^{-1}\\
\end{pmatrix}$ where $A$ is lower-triangular and $\lambda \in R^\times$.
Denote by $B$ the Borel subgroup of $G$ defined by $B = I(4)$.

For a quadratic extension $L$ of $\Q$ and $v$ be a finite place of $\Q$, define $L_v =L\otimes_\Q \Q_v$.

$\Z_L$ denotes the ring of integers of $L$ and $\Z_{L,v}$ its $v$-closure in $L_v$.For a prime $p$, let
$\Z_{L,p}^\times$ denote the group of units in $\Z_{L,p}$.

If $p$ is inert in $L$, the elements of $\Z_{L,p}^\times$ are of
the form  $a + b\sqrt{-d}$ with $a,b \in \Z_p$ and such that at least one of $a$ and $b$ is a unit. Let
$\Gamma_{L,p}^0$ be the subgroup of $\Z_{L,p}^\times$ consisting of the elements with $p | b$.

For a positive integer $N$ the subgroups $\Gamma_0(N)$ and $\Gamma^0(N)$ of $SL_2(\Z)$ are defined by $$ \Gamma_0(N) =\{ A \in SL_2(\Z) \mid A \equiv \begin{pmatrix}
\ast & \ast\\
0 & \ast
\end{pmatrix}\pmod{N} \},$$
$$ \Gamma^0(N) =\{ A \in SL_2(\Z) \mid A \equiv \begin{pmatrix}
\ast & 0\\
\ast & \ast
\end{pmatrix}\pmod{N} \}.$$
For $p$ a finite place of $\Q$, their local analogues $\Gamma_{0,p}$ (resp. $\Gamma^0_p$) are defined by $$
\Gamma_{0,p}= \{ A \in GL_2(\Z_p) \mid A \equiv \begin{pmatrix}
\ast & \ast\\
0 & \ast
\end{pmatrix}\pmod{p} \},$$
$$ \Gamma^0_p =\{ A \in GL_2(\Z_p) \mid A \equiv \begin{pmatrix}
\ast & 0\\
\ast & \ast
\end{pmatrix}\pmod{p} \}.$$

The local Iwahori subgroup $I_p$
 is defined to be the subgroup of $K_p= G(\Z_p)$ consisting of
those elements of $K_p$ that when reduced mod $p$ lie in the Borel subgroup of $G(\mathbb{F}_p)$. Precisely,
$$I_p = \{ A \in K_p \mid A \equiv \begin{pmatrix}
\ast & 0 & \ast & \ast\\
\ast & \ast & \ast & \ast\\
0 & 0 & \ast & \ast\\
0 & 0 & 0  & \ast\\
\end{pmatrix}\pmod{p} \}$$

Let $\tR$ denote the subgroup of $\G \times \widetilde{F}$ consisting of elements $h=(h_1,h_2)$ such
that $h_1\in \G, h_2\in  \widetilde{F}$ and $\mu_2(h_1) =\mu_1(h_2)$. Let $R$ denote the subgroup of $\tR$ consisting of those $(h_1,h_2)$ where $h_1 \in G$.

For a fixed element $g \in \G(\A)$, let
$\tF_1[g](\A)$ denote the subset of $\tF(\A)$ consisting of all elements $h_2$ such that $\mu_2(g) =\mu_1(h_2)$.

\section{Eisenstein series on $GU(3,3)$}\label{section1}
\subsection{Assumptions}Let $S$ be a finite subset (possibly empty) of the finite places of $\Q$. Let $S_1, S_2, S_3$ be disjoint subsets of $S$ such that $S = S_1 \sqcup S_2 \sqcup S_3$.

We let $M$ denote the product of primes in $S_1 \sqcup S_2$ and $N$ denote the product of primes in $S_2 \sqcup S_3$. Thus $M,N$ are positive, squarefree integers determined by $S_1,S_2,S_3$. Conversely, any choice of positive, squarefree integers $M,N$ uniquely determines $S_1, S_2, S_3$ since we have \begin{itemize}

\item $S_1$ is the set of primes that divide $M$ but not $N$.
\item $S_2$ is the set of primes that divide $\gcd(M,N)$.
\item $S_3$ is the set of primes that divide $N$ but not $M$.

\end{itemize}

Let $L$ denote an an imaginary quadratic field such that all primes in $S_1 \sqcup S_2$ are inert in $L$. Fix an unitary character $\Lambda = \prod_v \Lambda_v$ of $L^\times(\A)/L^\times$ such that: \begin{enumerate}
\item $\Lambda | \A^\times =1$.

\item $\Lambda_\infty$ is trivial.
\item $\Lambda$ is unramified outside $S_1 \sqcup S_2$.

\item If $p \in S_1 \sqcup S_2$, then $\Lambda_p$ is non-trivial on $\Z_{L,p}^\times$ but trivial on the subgroup $\Gamma^0_{L,p}$.
 \end{enumerate}

\textbf{Remark:} Too see that such characters exist, we note that for each prime $q \in S_1 \sqcup S_2$, $L^\times \A^\times L_\infty (\prod_{p \neq q}\Z_{L,p}^\times) \Gamma^0_{L,q}$ is a subgroup of $L^\times L_\infty \A^\times (\prod_{p }\Z_{L,p}^\times)$ of index $\frac{q+1}{t}>1$. Here $2t$ is the cardinality of the group $\Z_L^\times$. For details, see~\cite[Subsection 9.3]{lfshort}.

\subsection{Eisenstein series}

Let $P_{\tH} = M_{\tH}N_{\tH}$ be the Siegel parabolic of $\tH$, with $$ M_{\tH}(\Q) :=
\left\{ m(A,v) = \begin{pmatrix}A & 0\\0&v \cdot (\overline{A^{-1}})^t\end{pmatrix} | A \in GL_3(L), v \in
\Q^\times \right\},$$ $$N_{\tH}(\Q):= \left\{n(b) = \begin{pmatrix}1 & b\\0&1 \end{pmatrix} | b \in M_3(L),
\overline{b}^t =b \right\}.$$

For $s \in \C,$ we form the induced representation $$I(\Lambda,s) = \otimes_vI_v(\Lambda_v,s)=
\text{Ind}_{P_{\tH}(\A)}^{\tH(\A)} (\Lambda\delta^{s})$$ consisting of smooth functions $\Xi$ on
$\tH(\A)$ such that \begin{equation}\label{e:upsilondefformula}\Xi(nm(A,v)g,s) = |v|^{-9(s +
\frac{1}{2})}|N_{L/\Q}(\det A) |^{3(s + \frac{1}{2})}\Lambda(\det A)\Xi(g,s)\end{equation} for $n\in
N_{\tH}(\A)$, $m(A,v)\in M_{\tH}(\A)$, $g\in \tH(\A)$. Here $\delta$ denotes the modulus character of $P_{\tH}$.

Finally, given such a section $\Xi$, we form the Eisenstein series $E_{\Xi}(h,s)$ by
\begin{equation}\label{def:eisenstein}E_{\Xi}(h,s)= \sum_{\gamma \in P_{\tH}(\Q) \backslash \tH(\Q)} \Xi(\gamma h,s)\end{equation} for $\Re(s)$ large, and defined elsewhere by meromorphic continuation.

\subsection{Some compact subgroups}For each finite place $p$ of $\Q$, define the maximal compact subgroups $K_p^{\tH}, K_p^{\G}, K_p^{\tF}$ of (respectively) $\tH(\Q_p)$, $\G(\Q_p)$, $\tF(\Q_p)$ by $$K_p^{\tH} = \tH(\Q_p) \cap GL_6(\Z_{L,p}),$$
 $$K_p^{\G} = \G(\Q_p) \cap GL_4(\Z_{L,p}),$$ $$K_p^{\tF} = \tF(\Q_p) \cap GL_2(\Z_{L,p}).$$

 Let $U_p^{\tH}$ be the subgroup of $K_p^{\tH}$ defined by $$U_p^{\tH} =
\big\{ z \in K_p^{\tH} \mid z \equiv
\begin{pmatrix}
 \ast &  \ast &  \ast & \ast &\ast &\ast\\
 \ast &  \ast &  \ast &\ast &\ast & \ast\\
 \ast & \ast & \ast & \ast & \ast & \ast\\ \ast & \ast & \ast & \ast & \ast & \ast\\
0  & 0  & 0   & 0 &\ast & \ast\\
0  & 0 & 0   & 0 &\ast &\ast
\end{pmatrix} \pmod{p} \big\}.$$

Let $r : K_p^{\tH} \rightarrow \tH(\F_p)$ be the canonical map and define the subgroup $$I_p'^{\tH} = r^{-1}I(6,\F_p).$$

Also, put $$K_\infty^{\tH} = \{ g \in \tH(\R) | \mu_3(g) =1, g(iI_3) = i I_3 \},$$ $$K_{\infty}^{\G} = \{ g
\in \G(\R) | \mu_2(g) =1, g(iI_2) = i I_2 \}$$ and $$K_\infty^{\tF} = \{ g \in \tF(\R) | \mu_1(g) =1, g(i) = i \}.$$

By \cite[p.5]{ich}, any matrix $k_\infty$ in $K_\infty^{\tH}$ (resp. $K_\infty^{\G}, K_\infty^{\tF}$) can be written in the form $k_\infty= \lambda\begin{pmatrix}A&B\\-B&A\end{pmatrix}$ where $\lambda \in  \C, |\lambda|=1$, and $A+iB,A-iB$ lie in $U(3;\R)$ (resp. $U(2; \R), U(1;\R)$) with $\det(A+iB) = \overline{\det(A-iB)}.$

For a positive even integer $\ell$, define \begin{equation}\label{e:defrhol2}\rho_\ell(k_\infty)= \det(A-iB)^{-\ell}.\end{equation}
Note that an alternate definition for $\rho_\ell(k_\infty)$ is simply $$\rho_\ell(k_\infty)=\det(k_\infty)^{\ell/2}\det(J(k_\infty,i))^{-\ell}.$$
Also note that if $k_\infty$ has all real entries, then $$\rho_\ell(k_\infty)=\det(J(k_\infty,i))^{-\ell}.$$

\subsection{A particular choice of section}\label{partichoice}
Fix an element  $Q \in H_1(\Z)$ and an element $\Omega \in \tH_1(\Z)$.
We abuse notation and use $Q, \Omega$ to also denote their natural inclusions into $\tH(\Q_v)$ for any place $v$.

We impose the following condition on $\Omega$ for all primes $p \in S_2$:

\emph{If $nm(A,v) \in P_{\tH}(\Q_p) \cap \Omega I_p'^{\tH} \Omega^{-1}$, then $\det(A) \in \Gamma_{L,p}^0$.}

We next define, for each place $v$, a particular section $\Upsilon_v(s) \in
I_v(\Lambda_v,s) $.

Recall that $I_v(\Lambda_v,s) $ consists of smooth functions $\Xi$ on
$\tH(\Q_v)$ such that \begin{equation}\Xi(nm(A,t)g,s) = |t|_v^{-9(s +
\frac{1}{2})}|N_{L/\Q}(\det A) |_v^{3(s + \frac{1}{2})}\Lambda_v(\det A)\Xi(g,s)\end{equation} for $n\in
N_{\tH}(\Q_v)$, $m(A,t)\in M_{\tH}(\Q_v)$, $g\in \tH(\Q_v)$.

\begin{itemize}

\item Clearly $I_p(\Lambda_p,s)$ has a $K_p^{\tH}$ fixed vector whenever $\Lambda_p$ is unramified.

For all finite places $p \notin S$, choose $\Upsilon_p$ to be the unique $K_p^{\tH}$ fixed vector with
\begin{equation}\label{e:upsilonfin}\Upsilon_p(1,s) = 1.\end{equation}

\item For all finite places $p \in S_3$, choose $\Upsilon_p$ to be the unique $U_p^{\tH}$ fixed vector with
\begin{equation}\label{e:upsilonfin2}\Upsilon_p(Q,s) = 1\end{equation} and $$\Upsilon_p(t,s) = 0$$ if $t \notin P_{\tH}(\Q_p)QU_p^{\tH}$.

\item Suppose $p\in S_2$.
 Choose $\Upsilon_{p}$ to be the unique
 $I_{p}'^{\tH}$ fixed vector with
\begin{equation}\label{e:upsilonfin3}\Upsilon_p(Q,s) = 1\end{equation}
 and $$\Upsilon_{p}(t,s) = 0$$ if $t \notin P_{\tH}(\Q_{p})QI_{p}'^{\tH}$.
We note here that such a well-defined vector as above exists because if $nm(A,v) \in P_{\tH}(\Q_p) \cap Q I_p'^{\tH} Q^{-1}$, then $\det(A) \in \Gamma_{L,p}^0$. This follows because $Q\cdot\in H_1(\Z)$.

\item  Let $p\in S_1$.

Choose $\Upsilon_p$ to be the unique
 $I_p'^{\tH}$ fixed vector with
\begin{equation}\label{e:upsilonfin4}\Upsilon_p(\Omega,s) = 1, \qquad \Upsilon_p(Q,s) = 1\end{equation}
 and $$\Upsilon_p(t,s) = 0$$ if $t \notin P_{\tH}(\Q_p)\Omega I_p'^{\tH} \sqcup P_{\tH}(\Q_{p})QI_{p}'^{\tH}$. It is an easy exercise to check that such a vector exists by our assumption on $\Omega$.

\item Finally choose $\Upsilon_\infty$ to be the unique vector in
$I_\infty(\Lambda_\infty,s)$ such that
\begin{equation}\label{e:upsiooninf}
\Upsilon_\infty(k_\infty,s) = \rho_\ell(k_\infty)\end{equation} for $k_\infty \in K_\infty^{\tH}$.

\end{itemize} \bigskip
Let $\Upsilon$ be the factorizable section in $\text{Ind}_{P_{\tH}(\A)}^{\tH(\A)} (\Lambda \| \cdot \|^{3s})$.
defined by $$\Upsilon(s) = (\otimes_v \Upsilon_v(s)).$$

As explained in~\eqref{def:eisenstein}, this gives rise to an Eisenstein series $E_{\Upsilon}(g,s)$.

Note that $\Upsilon$ and $E_{\Upsilon}$ are right invariant by $\prod_{p\in S_1 \sqcup S_2}I_p'^{\tH}\prod_{ p \in S_3 }U_p^{\tH}\prod_{\substack{p \notin S\\p<\infty}}K_p^{\tH}.$

\section{Statement of the pullback formula}\label{s:pullbackstatement}

\subsection{Assumptions}
 \emph{For the rest of this paper, we assume that all primes in $S$ are odd and inert in $L$.}

  Let $a, b$ be integers and $d$ a positive integer such that $L= \Q(\sqrt{-d})$ and $-d = b^2-4a.$

\emph{Also, we henceforth fix $Q$ to equal the following matrix:
$$Q=\begin{pmatrix}
0&1&0&0&0&0\\1&0&0&0&0 &0\\0&0&0&0 &0&-1
\\0&0&0&0&1&-1\\0&0&0&1&0&0\\0&1&1&0&0&0\end{pmatrix}.$$}

\emph{Further, define the element $\Theta \in \G_1(\Z)$ by } $$\Theta=\begin{pmatrix}
1 & 0 & 0 & 0\\
\alpha & 1 & 0 & 0\\
0 & 0& 1 & -\overline{\alpha}\\
0 & 0 & 0 & 1\\
\end{pmatrix} \text{ where } \alpha = \frac{b + \sqrt{-d}}{2},$$

\emph{and the element $s_1 \in G_1(\Z)$ by  } $$s_1=\begin{pmatrix}0&1&0&0\\1&0&0&0\\0&0&0&1\\0&0&1&0\end{pmatrix}.$$

\subsection{Eisenstein series on $GU(2,2)$}\label{s:eisensteindef}

Let $P$ be the maximal
parabolic subgroup of $\widetilde{G}$ consisting of the elements in $\G$ that look like$
\begin{pmatrix}
\ast & \ast & \ast & \ast\\
0 & \ast & \ast & \ast\\
0 & \ast & \ast & \ast\\
0 & \ast & \ast & \ast\\
\end{pmatrix}$. We have the Levi
decomposition $P= MN$ with  $M= M^{(1)}M^{(2)}$ where the groups $M, N, M^{(1)}, M^{(2)}$ are as defined in
\cite{fur}.

Precisely,
\begin{equation}
M^{(1)}(\Q) =\left\{\begin{pmatrix}
a & 0 & 0 & 0\\
0 & 1 & 0 & 0\\
0 & 0 & \overline{a}^{-1} & 0\\
0 & 0 &0  & 1\\
\end{pmatrix} \mid a \in L^\times \right \} \simeq L^\times.
\end{equation}

\begin{equation}\begin{split}
M^{(2)}(\Q) &=\left\{\begin{pmatrix}
1 & 0 & 0 & 0\\
0 & \alpha & 0 & \beta\\
0 & 0 & \lambda & 0\\
0 & \gamma
 &0  & \delta\\
\end{pmatrix} | \begin{pmatrix} \alpha & \beta\\ \gamma & \delta \end{pmatrix} \in GU(1,1)(\Q), \lambda =
\mu_1 \begin{pmatrix} \alpha & \beta\\ \gamma & \delta
\end{pmatrix}\right \} \\ &\simeq GU(1,1)(\Q).\end{split}
\end{equation}

\begin{equation}
N(\Q) =\left\{\begin{pmatrix}
1 & x & 0 & 0\\
0 & 1 & 0 & 0\\
0 & 0 & 1 & 0\\
0 & 0 & -\overline{x}  & 1\\
\end{pmatrix}\begin{pmatrix}
1 & 0 & a & y\\
0 & 1 & \overline{y} & 0\\
0 & 0 & 1 & 0\\
0 & 0 & 0 & 1\\
\end{pmatrix} \mid a \in \Q , x,y \in L \right \} .
\end{equation}

We also write $$m_1(a) = \begin{pmatrix}
a & 0 & 0 & 0\\
0 & 1 & 0 & 0\\
0 & 0 & \overline{a}^{-1} & 0\\
0 & 0 &0  & 1\\
\end{pmatrix},$$ $$m_2\begin{pmatrix} \alpha & \beta\\ \gamma & \delta \end{pmatrix}= \begin{pmatrix}
1 & 0 & 0 & 0\\
0 & \alpha & 0 & \beta\\
0 & 0 & \lambda & 0\\
0 & \gamma &0  & \delta\\
\end{pmatrix}.$$

Next, let $ g$ be a normalized newform of weight $\ell$ for $\Gamma_0(N)$. $ g$ has a Fourier expansion $$ g(z) = \sum_{n=1}^\infty b(n)e(nz)$$ with $b(1) =1$. It is then well known that the $b(n)$ are all totally real algebraic numbers.

We define a function $\Psi$ on $GL_2(\A)$ by $$\Psi(\gamma g_\infty k_0) = (\det
g_\infty)^{\frac{\ell}{2}}(c i + d)^{-\ell} g(g_\infty(i))$$ where $\gamma \in GL_2(\Q), \ g_\infty= \begin{pmatrix}a
&b\\ c & d\end{pmatrix} \in GL_2^+(\R),$ and $$k_0 \in \prod_{p \notin S_2 \cup S_3}GL_2(\Z_p)\prod_{p \in S_2 \cup S_3}\Gamma_{0,p} .$$

Let $\sigma$ be the automorphic representation of
$GL_2(\A)$ generated by $\Psi$.
We know that $\sigma =\otimes \sigma_v$ where $$\sigma_v =  \begin{cases}\text{holomorphic discrete series} & \text{ if } v=\infty,\\
 \text{unramified spherical principal series} &\text{ if } v \text{ finite }, v \nmid N, \\
 \xi \text{St}_{GL(2)}  \text{where } \xi_v \text{ unramified, } \xi_v^2=1 &\text{ if } v\mid N. \end{cases}$$

 If $p \nmid N$, we let $\alpha_p,  \beta_p$ denote the unramified characters of $\Q_p^\times$ that induce the spherical local representation $\sigma_p$.

 For a prime $p$, let $r: K_p^{\G} \rightarrow \G(\F_p)$ be the canonical map and define the subgroup $$I_p' = r^{-1}I(4, \F_p).$$

 Also, let $U_p^{\G}$ be the subgroup of $K_p^{\G}$ defined by $$U_p^{\G} =
\big\{ z \in K_p^{\G} \mid z \equiv
\begin{pmatrix}
\ast & 0 & \ast & \ast\\
\ast & \ast & \ast & \ast\\
\ast & 0 & \ast & \ast\\
0 & 0 & 0  & \ast\\
\end{pmatrix} \pmod{p} \big\}.$$

Extend $\Psi$ to $\tF(\A)$ by $$\Psi(ag) = \Psi(g)$$ for $a \in
L^\times(\A), g \in GL_2(\A).$
Now define the compact open subgroup $U^{\G}$ of $\G(\A_f)$ by
\begin{equation}\label{def:kgtilde}U^{\G}=\prod_{p \notin S}K_p^{\G}\prod_{ p \in S_3}U_p^{\G}\prod_{ p \in S_1 \cup S_2} I_p'\end{equation}

 Define \begin{equation}\label{defflambda}f_\Lambda(g,s) = \delta_P^{s+\frac{1}{2}}(m_1m_2)\Lambda(\overline{m_1})^{-1}
\Psi(m_2)
\rho_l(k_\infty) \quad \text{if } g=m_1m_2n\widetilde{k}k \in \G(\A)\end{equation} where
 $m_i \in M^{(i)}(\A) \quad (i=1,2)$, $n\in N(\A),$  $
 k=k_\infty k_0$ with  $k_\infty \in K_{\infty}^{\G}$, $k_0
\in U^{\G} $ and $\widetilde{k} = \prod_p k_p \in \prod_pK_p^{\G}$ is such that $k_p= 1$ if $p \notin S_1 \sqcup S_2$, $k_p \in \{1,s_1\}$ for $p \in S_2$ and $k_p \in \{1,\Theta \}$ for $p \in S_1$. Put $$ f_\Lambda(g,s) = 0$$ if $g$ is not of the form above. It can be easily verified that everything is well-defined.

We define the Eisenstein series $E_{\Psi,\Lambda}(g,s)$ on $\widetilde{G}(\A)$ by
\begin{equation}E_{\Psi,\Lambda}(g,s) = \sum_{\gamma \in P(\Q) \backslash \G(\Q)}f_\Lambda(\gamma g,s).
\end{equation}

\subsection{An important embedding}\label{s:embeddingdef}
We define an embedding $\iota:\tR \hookrightarrow \tH$ by \begin{equation}\label{e:embedmatrix}\i: (\mat{A}{B}{C}{D},\mat{a}{b}{c}{d})
\longmapsto
  \begin{bmatrix}A&&B\\&a&&-b\\C&&D\\&-c&&d\end{bmatrix}.\end{equation} 

An essential feature of this embedding is the following. Suppose $$g_1=m_1(a)m_2(b)n \in P(\A),$$ $$g_2=b$$
where $$m_1(a)=\begin{pmatrix}
a & 0 & 0 & 0\\
0 & 1 & 0 & 0\\
0 & 0 & \overline{a}^{-1} & 0\\
0 & 0 &0  & 1\\
\end{pmatrix} \in M^{(1)}(\A),$$
$$n \in N(\A),\ b=\begin{pmatrix}\alpha&\beta\\ \gamma &\delta\end{pmatrix} \in \tF(\A),$$
and $$m_2(b)=\begin{pmatrix}
1 & 0 & 0 & 0\\
0 & \alpha & 0 & \beta\\
0 & 0 & \lambda & 0\\
0 & \gamma &0  & \delta\\
\end{pmatrix} \in M^{(2)}(\A),$$ where $\lambda=\mu_1 \begin{pmatrix} \alpha & \beta\\ \gamma & \delta
\end{pmatrix}$.
 Then
\begin{equation}\label{e:keyparabolicfact}Q\cdot\iota(g_1,g_2)Q^{-1} \in P_{\tH}(\A).\end{equation} It is this key fact that enables us
to pass from Klingen Eisenstein series on $\G(\A)$ to Siegel Eisenstein series on $\tH(\A)$.

\emph{Henceforth, we fix $$\Omega= Q\cdot\i(\Theta, 1).$$}
We note that $\Omega$ satisfies the condition stated at the beginning of Subsection~\ref{partichoice}.

 \subsection{The Pullback formula}\label{s:defpullback}
 For an element $g \in \G(\A)$, let
$\tF_1[g](\A)$ denote the subset of $\tF(\A)$ consisting of all elements $h_2$ such that $\mu_2(g) =\mu_1(h_2)$. Clearly $\widetilde{F_1}(\Q)$ acts on $\tF_1[g](\A)$ by left multiplication.

We will compute the integral \begin{equation}\mathcal{E}(g,s)=\int_{  \widetilde{F_1}(\Q) \backslash
\tF_1[g](\A)}E_{\Upsilon}(\i(g,h),s)\Psi(h)\Lambda^{-1}(\det h) dh.\end{equation}

Here, the measure is normalized by making all the local maximal compact subgroups $K_v^{\tF}$ have measure 1. Define $$\zeta^S(s)=\prod_{p \notin S}(1-p^{-s})^{-1},$$ $$L^S(s, \chi_{-D})=\prod_{\substack{p \notin S \\ \gcd(p, D)=1}}(1-(\chi_{-D})_p(p)p^{-s})^{-1}$$ where $\chi_{-D}$ denotes the character of $\A^\times$ associated to $L$.

Also, let $\rho(\Lambda)$ denote the representation of $GL_2(\A)$ obtained from $\Lambda$ by automorphic induction. Hence, for a prime $q \notin S$, we have:

$L(s, \sigma_q \times \rho(\Lambda_q)) \\= \begin{cases}
(1-\alpha^2(q)q^{-2s})^{-1}(1-\beta^2(q)q^{-2s})^{-1} &\text{if $q$ is inert in $L$, }\\
\\(1-\alpha(q)\Lambda_q(q_1) q^{-s})^{-1}(1-\beta(q)\Lambda_q(q_1) q^{-s})^{-1} & \text{if $q$ is ramified
in $L$,}\\ \\(1-\alpha(q)\Lambda_q(q_1) q^{-s})^{-1}(1-\beta(q)\Lambda_q(q_1)
q^{-s})^{-1}\\
\cdot(1-\alpha(q)\Lambda_q^{-1}(q_1) q^{-s})^{-1}(1-\beta(q)\Lambda_q^{-1}(q_1) q^{-s})^{-1}& \text{if $q$
splits in $L$,}
\end{cases}
$

where $q_1 \in \Z_{L,q}$ is any element with $N_{L/\Q}(q_1) \in q\Z_q^\times$.

Also for a prime $p \in S_3$, put $$L(s, \sigma_p \times \rho(\Lambda_p)) = (1-p^{-2s-1})^{-1}.$$

Put $$L(s, \sigma \times \rho(\Lambda)) = \prod_{q\nmid M} L(s, \sigma_q \times \rho(\Lambda_q)).$$

Now define \begin{equation}\label{defB(s)}B(s) = \frac{B_\infty(s)L(3s+1, \sigma \times \rho(\Lambda))}{\sigma_1(M)^2\sigma_1(N/\gcd(M,N))P_{S_3} L^S(6s+2, \chi_{-D}) \zeta^S(6s+3)}\end{equation} where $$\sigma_1(A) = \prod_{p |A} (p+1)$$ and $$B_\infty(s)=\frac{ (-1)^{\ell/2}2^{-6s -1} \pi }{6s+\ell -1}.$$

Then the pullback formula says:

\begin{theorem}[Pullback formula]\label{t:pullback} For $g \in \G(\A)$ define $\mathcal{E}(g,s)$ as above and $E_{\Psi,\Lambda}(g,s)$ as in Subsection~\ref{s:eisensteindef}. Then we have $$\mathcal{E}(g,s)=B(s)E_{\Psi,\Lambda}(g,s) $$ as an identity of meromorphic functions.
\end{theorem}

We will prove the Pullback formula in Section~\ref{s:pullback} using the machinery developed in the next two sections.

\section{The local integral and the unramified calculation}\label{s:section3}

\subsection{Definitions} We retain the notations and definitions of the previous section. Furthermore, for any prime $p$, we define the following compact subgroups of $\tF(\Q_p)$:

\begin{itemize}

\item $\Gamma_{0,p}^{\tF}= \{ A \in K_p^{\tF} \mid A \equiv \begin{pmatrix}
\ast & \ast\\
0 & \ast
\end{pmatrix}\pmod{p} \}$

\item Let $r_p:K_p^{\tF} \rightarrow GU(1,1)(\F_p)$ be the canonical map and let $K_p'^{\tF} = r_p^{-1}(GL_2(\F_p))$. Define $$\Gamma_{0,p}'^{\tF} = K_p'^{\tF} \cap \Gamma_{0,p}^{\tF}.$$

\end{itemize}

\subsection{Some useful properties}\label{s:sectionproperties} First, we note some properties of the section $\Upsilon$. Fix  $(g_1,g_2) \in \widetilde{R}(\A)$.
\begin{itemize}

\item Let $p$ be a prime not dividing $MN$ and  $k_1 \in K_p^{\G}$, $k_2\in K_p^{\tF}$ with
$\mu_2(k_1)=\mu_1(k_2)$. Then, note that $$\i(k_1,k_2)\in K_p^{\tH}.$$ Because $\Upsilon_p$ is $K_p^{\tH}$-fixed, it follows that
\begin{equation}\label{e:upsilonunramified}
\Upsilon(\i(g_1k_1,g_2k_2),s)=\Upsilon(\i(g_1,g_2),s),\end{equation}

\item Let $p|N, p \nmid M$. If $k_1 \in U_p^{\G}$, $k_2\in \Gamma_{0,p}^{\tF}$ with
$\mu_2(k_1)=\mu_1(k_2)$ then check that \begin{equation}\label{e:qcrucila2}\i(k_1, k_2) \in U_p^{\tH}.\end{equation} Because $\Upsilon_p$ is $U_p^{\tH}$-fixed, it follows that
\begin{equation}\label{e:upsilonpdivN}
\Upsilon(\i(g_1k_1,g_2k_2),s)=\Upsilon(\i(g_1,g_2),s),
\end{equation}

\item Let $p$ be a prime dividing $M$. If $k_1 \in I_p'$, $k_2\in \Gamma_{0,p}'^{\tF}$ with
$\mu_2(k_1)=\mu_1(k_2)$ then check that \begin{equation}\label{e:qcrucila3}\i(k_1, k_2) \in I_p'^{\tH}.\end{equation} Because $\Upsilon_p$ is $I_p'^{\tH}$-fixed, it follows that
\begin{equation}\label{e:upsilonpdivM}
\Upsilon(\i(g_1k_1,g_2k_2),s)=\Upsilon(\i(g_1,g_2),s),
\end{equation}

\item Finally, let $k_1 \in K_{\infty}^{\G}$, $k_2 \in K_\infty^{\tF}$ with
$\mu_2(k_1)=\mu_1(k_2)$. Check that \begin{equation}\label{e:qcrucilainf}\i(k_1, k_2) \in K_\infty^{\tH}.\end{equation} Hence we have \begin{equation}
\label{e:upsiloninf}
\Upsilon(\i(g_1k_1,g_2k_2),s)=\rho_\ell(k_1)
\rho_\ell(k_2)^{-1}
\Upsilon(\i(g_1,g_2),s).\end{equation}

\end{itemize}

\subsection{The key local zeta integral}

Let $\psi =  \prod_v\psi_v$ be a character of $\A$ such that
\begin{itemize}
\item The conductor of $\psi_p$ is $\Z_p$ for all (finite) primes $p$,
\item $\psi_\infty(x) = e(x),$ for $x \in \R$,
\item $\psi|_\Q =1.$
\end{itemize}

Let $W_{\Psi}$ be the Whittaker model for $\Psi$. It is a function on $\tF(\A)$ defined by
$$
W_{\Psi}(g) = \int_{\Q \bs \A}\Psi\left(\begin{pmatrix}1&x\\0&1\end{pmatrix}g\right)\psi(-x)dx.$$ We have the Fourier expansion
\begin{equation}\Psi(g) = \sum_{\lambda \in \Q^\times}W_{\Psi}\left(\begin{pmatrix}\lambda&0\\0&1
\end{pmatrix}g\right)\end{equation}

By the uniqueness of Whittaker models, we have a factorization $$W_\Psi = \otimes_v W_{\Psi, v}.$$

Now, for each place $v$, and elements $g_v \in \tF(\Q_v), k_v \in K_v^{\G}$, define the local zeta integral \begin{equation}\label{e:deflocalzeta}Z_v(g_v,k_v,s) = \int_{  \tF_1(\Q_v)}\Upsilon_v(Q \cdot \i(k_v,h_v),s)
W_{\Psi,v}(g_vh_v)\Lambda_v^{-1}(\det h_v)dh_v,\end{equation}

The evaluation of this local integral at each place $v$ lies at the heart of our proof of the pullback formula.

First of all, by~\eqref{e:keyparabolicfact} and the properties proved in the previous subsection, observe that it is enough to evaluate the integral for $k_v$ lying in a fixed set of representatives of $(P(\Q_v)\cap K_v^{\G})\bs K_v^{\G} / U_v$, where $$U_v = \begin{cases} K_v^{\G} &\text{if } v \notin S\\ U^{\G}_v &\text{if } v \in S_3\\ I_v' &\text{if } v \in S_1 \sqcup S_2\end{cases}$$

For $1\le i \le 5$, define the matrices $s_i\in G(\Q)$  as follows:

$
s_1=\begin{pmatrix}0&1&0&0\\1&0&0&0\\0&0&0&1\\0&0&1&0\end{pmatrix}$
 ,  $\quad s_2=\begin{pmatrix}0&0&0&1\\0&0&1&0\\0&-1&0&0\\-1&0&0&0\end{pmatrix} ,$ $ \quad s_3=\begin{pmatrix}0&0&1&0\\0&0&0&1\\-1&0&0&0\\0&-1&0&0\end{pmatrix} ,$

 $s_4=\begin{pmatrix}0&-1&0&0\\-1&0&0&0\\0&0&0&-1\\1&0&-1&0\end{pmatrix}
 ,$ $\quad s_5=\begin{pmatrix}0&-1&0&0\\-1&0&0&0\\0&1&0&-1\\0&0&-1&0\end{pmatrix}.$

Define the set $Y_\infty = \{1\}$ and for a (finite) prime $p,$ define the set $Y_p \subset \G(\Q_p)$ as follows:
\begin{itemize}
\item $Y_p = \{1 \}$ if $p \nmid MN$.
\item $Y_p= \{1,s_1,s_2\}$ if $p | N, p\nmid M$.
\item $Y_p= \{1,s_1,s_2, s_3, \Theta, \Theta s_2, \Theta s_4,\Theta s_5 \}$ if $p | M$.
\end{itemize}

\textbf{Remark}. In the above definition, we consider the $s_i$ and $\Theta$ as elements of $\G(\Q_p)$. This makes $Y_v$ a subset of $\G(\Z_v)$ for all places $v$.

\begin{lemma}\label{sirepresentatives}$Y_v$ is a set of representatives for $(P(\Q_v)\cap K_v^{\G})\bs K_v^{\G} / U_v$ at all places $v$.
\end{lemma}
\begin{proof}
For $v$ infinite or $v$ a prime not dividing $MN$, this is obvious.
Now let $p$ be a prime dividing $N$ but not $M$. If $W$ denotes the eight element Weyl group, then $W$ is a set of representatives for $(P(\Q_p)\cap K_p^{\G})\bs K_p^{\G} / I_p^{\G}$ where $I_p^{\G}$ denotes the Iwahori subgroup of $K_p^{\G}$. Since $U_p^{\G}$ is larger than $I_p^{\G}$, there is some collapsing, as expected. By explicit computation we find that $\{ 1, s_1, s_2 \}$ do form a set of distinct representatives.
The case when $p |M$ is also proved similarly by explicit computation. For brevity, we do not include the details here.
\end{proof}

The rest of this section and the next will be devoted to evaluating at each place $v$ the integral $Z_v(g_v, k_v, s)$ for every $k_v \in Y_v$, $g_v \in \tF(\Q_v)$.

\subsection{The local integral at unramified places}
In this subsection, $q$ will denote a prime that does not divide $MN$. Hence, both $\Lambda_q$ and $\sigma_q$ are unramified.

In particular, $\sigma_q$ is a spherical principal series representation induced from unramified characters $\alpha, \beta$ of $\Q_q^\times$.

By abuse of notation we use $q$ to also denote its inclusion in $\Q_q^\times$. Thus $q$ is an uniformizer in our local field.

Let $\rho(\Lambda)$ denote the representation of $GL_2(\A)$ obtained from $\Lambda$ by automorphic induction. Define $L(s, \sigma_q \times \rho(\Lambda_q))$ as in Subsection~\ref{s:defpullback}.

For a character $\chi$ of $\Q_q^{\times}$ define $L(s,\chi)=\begin{cases}(1-\chi(q)q^{-s})^{-1}& \text{if } \chi \text{ is unramified at } q,\\1 &\text{otherwise}.\end{cases}$

The aim of this subsection is to prove the following proposition.

\begin{proposition}\label{p:pullbackunramified}Let $q$ be a prime such that $q\nmid MN$. Let $\mathbf{1}$ denote the trivial character and $\chi_{-D}$ denote the Hecke character associated to the quadratic extension $L / \Q$. Then, we have $$Z_q(g_q,1,s) = W_{\Psi,q}(g_q) \cdot \frac{L(3s +1, \sigma_q \times \rho(\Lambda_q))}{L(6s+2,(\chi_{-D})_q)L(6s+3,\mathbf{1})}.$$
\end{proposition}

\begin{proof}
Let $K_q^{\tF_1}$ denote the maximal compact subgroup of $\tF_1(\Q_q)$ defined by $$K_q^{\tF_1}=\tF_1(\Q_q) \cap GL_2(\Z_{L,q}).$$ Note that for $g \in \tF_1(\Q_q), k_1,k_2 \in  K_q^{\tF_1}$, we have using~\eqref{e:upsilondefformula}, \eqref{e:upsilonunramified} \begin{align*}\Upsilon_q(Q\cdot\i(1,k_1gk_2),s)&= \Upsilon_q(Q\cdot\i(m_2(k_1)m_2(k_1)^{-1},k_1gk_2),s)\\ &= \Upsilon_q(Q\cdot\i(m_2(k_1)^{-1},gk_2),s)\\
&=\Upsilon_q(Q\cdot\i(1,g),s)\end{align*} In other words $\Upsilon_q(Q\cdot\i(1,g),s)$ only depends on the double
coset $K_q^{\tF_1}gK_q^{\tF_1}.$

There are three distinct cases: $q$ can be inert, split or ramified in $L$. We consider each of these cases
separately.
\begin{case}$q$ is inert in $L$.
\end{case}
In this case, $L_q$ is a quadratic extension of $\Q_q$. We may write elements of $L_q$ in the form $a + b
\sqrt{-d}$ with $a,b \in \Q_q$; then $\Z_{L,q} = a +b \sqrt{-d}$ where $a,b \in \Z_q.$ Also note that $\Lambda_q$
is trivial.

We know (Cartan decomposition) that $$\tF_1(\Q_q) = \bigsqcup_{n\ge0}K_q^{\tF_1} A_n K_q^{\tF_1}$$ where $A_n =
\begin{pmatrix}q^n&0\\0&q^{-n}\end{pmatrix}.$ So~\eqref{e:deflocalzeta} gives us \begin{equation}\label{e:eisensteinzeta} Z_q(g_q,1,s) = \sum_{n \ge
0}\Upsilon_q(Q\cdot\i(1,A_n),s)\int_{K_q^{\tF_1} A_n K_q^{\tF_1}}W_{\Psi,q}(g_qh_q)dh_q.\end{equation}

Given an element $k \in K_q^{\tF_1}$ we can find $l \in \Z_{L,q}^\times$ such that $kl \in GL_2(\Z_q)$. It
follows that if $$GL_2(\Z_q)A_nGL_2(\Z_q) = \bigsqcup_i a_iGL_2(\Z_q),$$ where $a_i \in SL_2(\Z_q)$ then
$$K_q^{\tF_1}A_nK_q^{\tF_1}=\bigsqcup_i a_iK_q^{\tF_1}.$$ The importance of this observation is that we can use
the theory of Hecke operators for $GL_2$ to evaluate $\int_{K_q^{\tF_1} A_n K_q^{\tF_1}}W_{\Psi,q}(g_qh_q)dh_q$.

Recall that classically $T(q^k)$ denotes the Hecke operator corresponding to the set $GL_2(\Z_q)S_kGL_2(\Z_q)$
where $S_k$ comprises of the matrices of size $2$ with entries in $\Z_q$ whose determinant generates the ideal $(q^k)$. Also observe
that \begin{align*}GL_2(\Z_q)S_{2n}GL_2(\Z_q) &= \begin{pmatrix}q^n&0\\0&q^n\end{pmatrix}GL_2(\Z_q)A_nGL_2(\Z_q)
\\&\bigsqcup\begin{pmatrix}q&0\\0&q\end{pmatrix} GL_2(\Z_q)S_{2n-2}GL_2(\Z_q).\end{align*} So we have \begin{align}
\int_{K_q^{\tF_1} A_n K_q^{\tF_1}}W_{\Psi,q}(g_qh_q)dh_q&=\sum_i
W_{\Psi,q}(g_qa_i)\\ \label{e:whittakerhecke}&=(\beta_{2n}-\beta_{2n-2})W_{\Psi,q}(g_q)\end{align} where $\beta_k$ is the eigenvalue corresponding to $\Psi$ for the Hecke operator $T(q^k)$. We put $\beta_k = 0$ if $k<0$.

Using \cite[Proposition 4.6.4]{bump} we have \begin{equation}\label{e:whitakerformul}\beta_{k}=\frac{q^{k/2}(\alpha(q)^{k+1} -
\beta(q)^{k+1})}{\alpha(q)-\beta(q)}\end{equation} for $k\ge 0$.

On the other hand, using~\eqref{e:embedmatrix} we see that  $Q\cdot\i(1, A_n)Q^{-1}$  is the matrix $$C=\begin{pmatrix}1&0&0&0&0&0\\0&1&0&0&0&0
\\0&0&q^{-n}&0&0&0\\0&0&q^{-n}-1&1&0&0
\\0&0&0&0&1&0
\\1-q^n&0&0&0&0&q^n\end{pmatrix}$$
We can write $C=PK$ where $$P=\begin{pmatrix}q^n&0&0&0&
0&1\\0&
1&
0&0&0&0\\
0&0&1&q^{-n}&0&0\\0&0&0&q^{-n}&0&0\\0&0&0&0&1&0\\0&0&0&0&0&1\end{pmatrix}\in P_{\tH}(\Q_q)$$ and \begin{equation}\label{e:defineK}K=\begin{pmatrix}1&0&0&0&0&-1
\\0&1&0&0&0&0\\0&0&1&-1&0&0\\0&0&1-q^n&q^{n}&0&0\\
0&0&0&0&1&0\\1-q^n&0&0&0&0&q^n\end{pmatrix}.\end{equation}

So, by~\eqref{e:upsilondefformula} we have \begin{equation}\label{e:upsilonan}\Upsilon_q(Q\cdot\i(1,A_n),s)
=q^{-6n(s+1/2)}
\Upsilon_q(KQ,s)\end{equation}

Also $KQ\cdot\in K_q^{\tH}$, hence $\Upsilon_q(KQ,s)=1$.

So, by~\eqref{e:eisensteinzeta},\eqref{e:whittakerhecke},\eqref{e:whitakerformul},\eqref{e:upsilonan}, we have \begin{align*}Z_q(g_q,1,s) &=W_{\Psi,q}(g_q)\bigg[\sum_{n \ge
0}q^{-6n(s+1/2)}\frac{q^n(\alpha(q)^{2n+1} -\beta(q)^{2n+1})}{\alpha(q)-\beta(q)} \\&  \qquad \qquad \qquad - \sum_{n \ge 1}q^{-6n(s+1/2)}\frac{q^{n-1}(\alpha(q)^{2n-1} -\beta(q)^{2n-1})}{\alpha(q)-\beta(q)}\bigg] \\ &=W_{\Psi,q}(g_q)\frac{(1-q^{-6s-3})(1+q^{-6s-2})}{(1-\alpha(q)^2q^{-6s-2})(1-\beta(q)^2q^{-6s-2})}\\
&=W_{\Psi,q}(g_q) \cdot \frac{L(3s +1, \sigma_q \times \rho(\Lambda_q))}{L(6s+2,\chi_{-D})L(6s+3,\mathbf{1})}
\end{align*}

\begin{case}$q$ is split in $L$.
\end{case}
We can identify $L_q$ with $\Q_q \oplus \Q_q$ with $\Q_q$ embedded diagonally as $t \mapsto (t,t)$.

For $g \in GL_n(\Q_q)$ denote $g^\ast = J_n^{-1}(g^t)^{-1}J_n$. Note that for $n=2$, $g^\ast = \frac{g}{\det{g}}$. Now there is a natural isomorphism of $GL_n(\Q_q)$ into $U(n,n)(\Q_q)$ given by $g \mapsto (g,g^\ast)$. Thus specializing to the $n=2$ case, $g \mapsto (g, \frac{g}{\det{g}})$  takes $GL_2(\Q_q)$ isomorphically onto $\tF_1(\Q_q)$.

Define $A_{m,k}$ to be the image of $\begin{pmatrix}q^{m+k}&0\\0&q^m\end{pmatrix}.$

Thus $A_{m,k} =
\begin{pmatrix}(q^{m+k},q^{-m})&0\\0&(q^{m},q^{-m-k})\end{pmatrix}.$

The Cartan decomposition gives us $$\tF_1(\Q_q) = \bigsqcup_{\substack{k\ge0\\m\in \Z}}K_q^{\tF_1} A_{m,k} K_q^{\tF_1}.$$

Let $q_1$ denote the element $(q,1)\in L_q.$
So $N_{L/\Q}(q_1) = q  .$ For brevity, let us denote $\Lambda_q(q_1)$ by $\lambda$. Note that for any integer $m$, $$\Lambda_q(q^m, q^{-m}) = \lambda^{2m}.$$
Now, using~\eqref{e:deflocalzeta}, we have
\begin{equation}\label{e:eisensteinzetasplit} Z_q(g_q,1,s) = \sum_{\substack{k\ge0\\m\in \Z}}\Upsilon_q(Q\cdot\i(1,A_{m,k}),s)\lambda^{-4m-2k}\int_{K_q^{\tF_1} A_{m,k} K_q^{\tF_1}}W_{\Psi,q}(g_qh_q)dh_q.\end{equation}

 Using the above conventions, and the notation of the inert case, we have
\begin{align*}GL_2(\Z_q)S_{k}GL_2(\Z_q) &= \begin{pmatrix}q^{-m}&0\\0&q^{-m}\end{pmatrix}GL_2(\Z_q)A_{m,k}GL_2(\Z_q)
\\&\bigsqcup\begin{pmatrix}q&0\\0&q\end{pmatrix} GL_2(\Z_q)S_{k-2}GL_2(\Z_q).\end{align*}

So, we have
\begin{equation}\label{e:whittakerheckesplit}\int_{K_q^{\tF_1} A_{m,k} K_q^{\tF_1}}W_{\Psi,q}(g_qh_q)dh_q=(\beta_{k}-\beta_{k-2})W_{\Psi,q}(g_q)\end{equation}
where we put $\beta_k = 0$ if $k<0$.
Now $Q\cdot\i(1,A_{m,k})Q^{-1}$ is the matrix $C$ where $$C=\begin{pmatrix}1&0&0&0&0&0\\
0&1&0&0&0&0
\\0&0&q^m&0&0&0\\0&0&q^m-1&1&0&0
\\0&0&0&0&1&0
\\1-q^{m+k}&0&0&0&0&q^{m+k}\end{pmatrix}.$$

[Note that by $C$ we actually mean the pair $(C,C^\ast)$. This convention will be used throughout our treatment of the split case; thus the letters $P,K$ etc. are really a shorthand for $(P,P^\ast), (K,K^\ast)$ etc.]

First we consider the case $m \ge 0.$ We can write $C=PK$

where $$P=\begin{pmatrix}q^{m+k}&0&0&0&
0&1\\0&
1&
0&0&0&0\\
0&0&q^m&-q^m&0 &0\\0&0&0&1&0&0\\0&0&0&0&1&0\\0&0&0&0&0&1\end{pmatrix}$$ and $$K=\begin{pmatrix}1&0&0&0&0&-1\\0&1&0&0&0&0\\0&0&q^m&1&0&0
\\0&0&q^m-1&1&0&0\\
0&0&0&0&1&0\\1-q^{m+k}&0&0&0&0&q^{m+k}\end{pmatrix}$$

Since $P \in P_{\tH}(\Q_q)$ we have, using~\eqref{e:upsilondefformula} \begin{equation}\label{e:upsilonansplit}
\Upsilon_q(Q\cdot\i(1,A_{m,k}),s)=\lambda^{2m+k}
q^{-3(2m+k)(s+1/2)}
\Upsilon_q(KQ,s)\end{equation}

Since $KQ \in K_q^{\tH}$, $\Upsilon_q(KQ,s)=1$.

Thus when $m \ge 0$ we have \begin{equation}\label{e:upsilonsplit1}
\Upsilon_q(Q\cdot\i(1,A_{m,k}),s)=\lambda^{2m+k}
q^{-(6m+3k)(s+1/2)}.\end{equation}

Now suppose $0 \ge m \ge -k.$ For convenience we temporarily put $n=-m$. So $0 \le n \le k.$

Writing $C$ in the form $PK$ we verify that \begin{equation}\label{e:upsilonansplit2}
\Upsilon_q(Q\cdot\i(1,A_{m,k}),s)=\lambda^{-2n+k}
q^{-3k(s+1/2)}.
\end{equation}

So, when $-k
\le m \le 0$ we have \begin{equation}\label{e:upsilonsplit2}
\Upsilon_q(Q\cdot\i(1,A_{m,k}),s)=\lambda^{2m+k}
q^{-3k(s+1/2)}.\end{equation}

Finally, consider the case $ m \le -k.$ For convenience we again put $n=-m$. So $0 \le k \le n.$
By similar calculations as above, we find that \begin{equation}\label{e:upsilonsplit3}
\Upsilon_q(Q\cdot\i(1,A_{m,k}),s)=\lambda^{2m+k}
q^{(6m+3k)(s+1/2)}.\end{equation}

Substituting~\eqref{e:whitakerformul},\eqref{e:whittakerheckesplit},\eqref{e:upsilonsplit1},\eqref{e:upsilonsplit2},\eqref{e:upsilonsplit3} into~\eqref{e:eisensteinzetasplit} we obtain
\begin{align*}&Z_q(g_q,1,s) \\&=W_{\Psi,q}(g_q)\sum_{k=0}^\infty(\beta_k -\beta_{k-2})\bigg[\sum_{m=1}^\infty \lambda^{-2m-k}q^{(-6m-3k)(s+1/2)} \\ &\qquad \qquad + \sum_{m=-k}^0\lambda^{-2m-k}q^{-3k(s+1/2)} + \sum_{m=-\infty}^{-k-1}\lambda^{-2m-k}q^{(6m+3k)(s+1/2)}\bigg]
\\&=\frac{W_{\Psi,q}(g_q)(1-q^{-6s-3})(1-q^{-6s-2})}
{(1-\alpha(q)\lambda q^{-3s-1})(1-\beta(q)\lambda q^{-3s-1})(1-\alpha(q)\lambda^{-1} q^{-3s-1})(1-\beta(q)\lambda^{-1} q^{-3s-1})}\\
&=W_{\Psi,q}(g_q) \cdot \frac{L(3s +1, \sigma_q \times \rho(\Lambda_q))}{L(6s+2,\chi_{-D})L(6s+3,\mathbf{1})}\end{align*}

\begin{case}$q$ is ramified in $L$.
\end{case}
We largely revert to the notation of the inert case. Write elements of $L_q$ as $a + bq_1$ with $a,b \in \Q_q$ and $q_1$ an uniformizer in $L_q$, that is, $N_{L/\Q}(q_1) \in q\Z_q^\times.$ So $\Z_{L,q} = a+b q_1$ with $a,b \in \Z_q.$ Put $\lambda = \Lambda_q(q_1)$. We have $\lambda^2 =1$.

The Cartan decomposition takes the form $$\tF_1(\Q_q) = \bigsqcup_{n\ge0}K_q^{\tF_1} A_n K_q^{\tF_1}$$ where $A_n =
\begin{pmatrix}q_1^n&0\\0&q_1^{-n}\end{pmatrix}.$ So~\eqref{e:deflocalzeta} gives us \begin{equation}\label{e:eisensteinzeta3} Z_q(g_q,1,s) = \sum_{n \ge
0}\Upsilon_q(Q\cdot\i(1,A_n),s)\int_{K_q^{\tF_1} A_n K_q^{\tF_1}}W_{\Psi,q}(g_qh_q)dh_q.\end{equation}

Now, $$K_q^{\tF_1} A_n K_q^{\tF_1} = \begin{pmatrix}q_1^{-n}&0\\0&q_1^{-n}\end{pmatrix}K_q^{\tF_1} \begin{pmatrix}q^n &0\\0&1\end{pmatrix} K_q^{\tF_1}$$.

So, by the same argument as in the inert case, we have, \begin{equation}\label{e:whittakerheckeramified}\int_{K_q^{\tF_1} A_n K_q^{\tF_1}}W_{\Psi,q}(g_qh_q)dh_q=(\beta_{n}-\beta_{n-2})W_{\Psi,q}(g_q)\end{equation}
where, of course, we put $\beta_n=0$ for negative $n$.

Now $\i(1,A_n)$ is the same matrix as in the inert case with $q$ replaced by $q_1$. So the same choice of $P$ and $K$ work.

Thus, by \eqref{e:upsilondefformula} we have \begin{equation}\label{e:upsilonanram}\Upsilon_q
(Q\cdot\i(1,A_n),s)=\lambda^n q^{-3n(s+1/2)}
\end{equation}
Substituting~\eqref{e:whitakerformul},\eqref{e:whittakerheckeramified},
\eqref{e:upsilonanram} in \eqref{e:eisensteinzeta3} we have \begin{align*}Z_q(g_q,1,s) &=W_{\Psi,q}(g_q)\bigg[\sum_{n \ge
0}\lambda^n q^{-3n(s+1/2)}\frac{q^{n/2}(\alpha(q)^{n+1} -\beta(q)^{n+1})}{\alpha(q)-\beta(q)} \\&  \qquad \qquad \qquad - \sum_{n \ge 2}\lambda^n q^{-3n(s+1/2)}\frac{q^{n/2-1}(\alpha(q)^{n-1} -\beta(q)^{n-1})}{\alpha(q)-\beta(q)}\bigg] \\&=W_{\Psi,q}(g_q)\frac{(1-q^{-6s-3})}{(1-\alpha(q)\lambda q^{-3s-1})(1-\beta(q)\lambda q^{-3s-1})}\\
&=W_{\Psi,q}(g_q) \cdot \frac{L(3s +1, \sigma_q \times \rho(\Lambda_q))}{L(6s+2,\chi_{-D})L(6s+3,\mathbf{1})}
\end{align*}
(Note that $L(s,\chi_{-D}) = 1$ in this case)

This completes the proof.
\end{proof}

\section{The local integral for the ramified and infinite places}\label{s:section4}

\subsection{The local integral for primes in $S_3$}

Let $r$ be a prime dividing $N$ but not $M$. Note that $r$ is inert by our assumptions. In this section we will prove the following proposition.

\begin{proposition}\label{p:pullbackunramifiedsteinberg}We have

$Z_r(g_r,k_r,s) = \begin{cases}\frac{1}{r+1}W_{\Psi,r}(g_r) \cdot L\big(3s +1, \sigma_r \times \rho(\Lambda_r)\big) &\text{ if } k_r =1\\0 &\text{ if } k_r=s_1 \text{ or } s_2.\end{cases}$

where the local $L$-function $L(s, \sigma_r \times \rho(\Lambda_r))$ is defined by $$L(s, \sigma_r \times \rho(\Lambda_r))=(1-r^{-2s-1})^{-1}.$$
\end{proposition}
 \begin{proof}
Recall that $\sigma$ is the irreducible automorphic representation of $GL_2(\A)$ generated by
$\widetilde{\Psi}$. Let $\sigma_r$ be the local component of $\sigma$ at the place $r$. We know that $\sigma_r = \mathrm{Sp} \otimes \tau$ where $\mathrm{Sp}$ denotes the special (Steinberg) representation and
$\tau$ is a (possibly trivial) unramified quadratic character. We put $a_r = \tau(r)$, thus $a_r = \pm 1$ is the
eigenvalue of the local Hecke operator $T(r)$.

We first deal with the case $k_r =1$. Let $\Gamma_{0,r}^{\tF_1}$ denote the compact open subgroup of $\tF_1(\Q_r)$ defined by $$\Gamma_{0,r}^{\tF_1}=\Gamma_{0,r}^{\tF} \cap \tF_1(\Q_r).$$ Note that for $g \in \tF_1(\Q_r), k_1,k_2 \in  \Gamma_{0,r}^{\tF_1}$, we have using~\eqref{e:upsilondefformula}, \eqref{e:upsilonpdivN} \begin{equation}\label{e:doublecosetinvariance}\begin{split}
\Upsilon_r(Q\cdot\i(1,k_1gk_2),s)&= \Upsilon_r(Q\cdot\i(m_2(k_1)m_2(k_1)^{-1},k_1gk_2),s)\\ &= \Upsilon_r(Q\cdot\i(m_2(k_1)^{-1},gk_2),s)\\
&=\Upsilon_r(Q\cdot\i(1,g),s)\end{split}\end{equation} In other words $\Upsilon_r(Q\cdot\i(1,g),s)$ only depends on the double
coset $\Gamma_{0,r}^{\tF_1}g\Gamma_{0,r}^{\tF_1}.$

Because $r$ is inert in $L$, $L_r$ is a quadratic extension of $\Q_r$. We may write elements of $L_r$ in the form $a + b
\sqrt{-d}$ with $a,b \in \Q_r$; then $\Z_{L,r} = a +b \sqrt{-d}$ where $a,b \in \Z_r.$ Also note that $\Lambda_r$
is trivial.

We know (Bruhat-Cartan decomposition) that \begin{equation}\label{e:brucart}\begin{split}\tF_1(\Q_r) &= \Gamma_{0,r}^{\tF_1}\quad \cup \quad \Gamma_{0,r}^{\tF_1}w \Gamma_{0,r}^{\tF_1}\\ &\cup \quad \bigsqcup_{n>0} \Gamma_{0,r}^{\tF_1}A_n \Gamma_{0,r}^{\tF_1} \quad \cup \quad \bigsqcup_{n>0} \Gamma_{0,r}^{\tF_1}A_n w\Gamma_{0,r}^{\tF_1}  \\ &\cup \quad \bigsqcup_{n>0} \Gamma_{0,r}^{\tF_1}wA_n \Gamma_{0,r}^{\tF_1} \quad \cup\quad \bigsqcup_{n>0} \Gamma_{0,r}^{\tF_1}wA_n w\Gamma_{0,r}^{\tF_1}.\end{split}\end{equation} where $A_n =
\begin{pmatrix}r^n&0\\0&r^{-n}\end{pmatrix}$ and $w=\begin{pmatrix}0&1\\-1&0\end{pmatrix}$. So~\eqref{e:deflocalzeta} gives us \begin{equation}\label{e:eisensteinzetasteinb} \begin{split} Z_r(g_r,1,s) &= \Upsilon_r(Q\cdot\i(1,1),s)\int_{\Gamma_{0,r}^{\tF_1} }W_{\Psi,r}(g_rh_r)dh_r\\&+\Upsilon_r(Q\cdot\i(1,w),s)\int_{\Gamma_{0,r}^{\tF_1}w \Gamma_{0,r}^{\tF_1} }W_{\Psi,r}(g_rh_r)dh_r\\&+\sum_{n >
0}\Upsilon_r(Q\cdot\i(1,A_n),s)\int_{\Gamma_{0,r}^{\tF_1} A_n \Gamma_{0,r}^{\tF_1}}W_{\Psi,r}(g_rh_r)dh_r\\&+\sum_{n >
0}\Upsilon_r(Q\cdot\i(1,A_nw),s)\int_{\Gamma_{0,r}^{\tF_1} A_nw \Gamma_{0,r}^{\tF_1}}W_{\Psi,r}(g_rh_r)dh_r\\&+\sum_{n >
0}\Upsilon_r(Q\cdot\i(1,wA_n),s)\int_{\Gamma_{0,r}^{\tF_1} wA_n \Gamma_{0,r}^{\tF_1}}W_{\Psi,r}(g_rh_r)dh_r\\&+\sum_{n >
0}\Upsilon_r(Q\cdot\i(1,wA_nw),s)\int_{\Gamma_{0,r}^{\tF_1} wA_nw \Gamma_{0,r}^{\tF_1}}W_{\Psi,r}(g_rh_r)dh_r  .\end{split}\end{equation}

Now $W_{\Psi,r}$ is an eigenvector for the Iwahori-Hecke algebra, hence each of the integrals in~\eqref{e:eisensteinzetasteinb} evaluates to a constant multiple of $W_{\Psi,r}(g_r)$. Thus for some function $A(s)$ (not depending on $g_r$) we have $$Z_r(g_r,1,s)= A(s)W_{\Psi,r}(g_r).$$ We may normalize $W_{\Psi,r}(1)=1$; it follows that \begin{equation}\label{e:aszrel}Z_r(g_r,1,s)= Z_r(1,1,s)W_{\Psi,r}(g_r)\end{equation}

Given an element $k \in \Gamma_{0,r}^{\tF_1}$ we can find $l \in \Z_{L,q}^\times$ such that $kl \in \Gamma_{0,r}$. It
follows that if $$\Gamma_{0,r}A_n\Gamma_{0,r} = \bigsqcup_i a_i\Gamma_{0,r},$$ where $a_i \in SL_2(\Z_q)$ then
$$\Gamma_{0,r}^{\tF_1}A_n\Gamma_{0,r}^{\tF_1}=\bigsqcup_i a_i\Gamma_{0,r}^{\tF_1}.$$
From \cite[Lemma 4.5.6]{miybook}, we may choose $a_i = \begin{pmatrix}r^n&mr^{-n}\\0&r^{-n}\end{pmatrix}$ where $0 \le m <r^{2n}.$ Using the formula in \cite[Lemma 2.1]{grokud}, we have $W_{\Psi,r}(a_i) = r^{-2n}$ and hence \begin{equation}\label{e:whittakerramifiedsum}\sum_{a\in \Gamma_{0,r}^{\tF_1} A_n \Gamma_{0,r}^{\tF_1}/\Gamma_{0,r}^{\tF_1}}W_{\Psi,r}(a) = 1\end{equation}

Also, from~\cite{miybook} we have $$\Gamma_{0,r}^{\tF_1}wA_nw\Gamma_{0,r}^{\tF_1}=\bigsqcup_i b_i\Gamma_{0,r}^{\tF_1}.$$ where $b_i = \begin{pmatrix}r^{-n} &0\\ -mr^{1-n}&r^n\end{pmatrix}$. Using the formula in \cite[Lemma 2.1]{grokud}, and doing some simple manipulations, we have \begin{equation}\label{e:whittakerramifiedsum2}\sum_{b\in \Gamma_{0,r}^{\tF_1} wA_nw \Gamma_{0,r}^{\tF_1}/\Gamma_{0,r}^{\tF_1}}W_{\Psi,r}(b) = 1\end{equation}

Next, we check that the quantities $\Upsilon_r(Q\cdot\i(1,A_nw),s)$, $\Upsilon_r(Q\cdot\i(1,wA_n),s),$
are both equal to $0$.
Indeed $\Upsilon_r(Q\cdot\i(1,A),s)=0$ whenever $Q\cdot\i(1,A)$ as an element of $\tH(\Q_r)$ does not belong to $P_{\tH}(\Q_r)QU_r^{\tH}$. Let $K$ be the matrix defined in~\eqref{e:defineK} with $q$ replaced by $r$. It suffices to prove that the quantities $KQ\cdot\i(m(w),1), KQ\cdot\i(1,w)\cdot $ do not belong to $(P_{\tH}(\Q_r)\cap K_r^{\tH})QU_r^{\tH}$. We check this by taking a generic element $P$ of  $(P_{\tH}(\Q_r)\cap K_r^{\tH})$ and showing that $Q^{-1}PK_0 \notin U_r^{\tH}$ where $K_0$ is one of the above quantities. That is a simple computation and is omitted.

On the other hand, putting $$P=\begin{pmatrix}r^n&0&0&0&
0&1\\0&
1&
0&0&0&0\\
0&0&1&r^{-n}&0&0\\0&0&0&r^{-n}&0&0\\0&0&0&0&1&0\\0&0&0&0&0&1
\end{pmatrix}\in P_{\tH}(\Q_r)$$
we can check that $$Q^{-1}P^{-1}Q\cdot \i(1,A_n) \in U_r^{\tH},$$ hence \begin{equation}\label{e:upsilonsteinformu}
\Upsilon_r(Q\cdot\i(1,A_n),s)=r^{-6n(s+1/2)}\end{equation}

Also, putting $$P=\begin{pmatrix}0&0&r^n&1&
0&0\\0&
1&
0&0&0&0\\
1&0&0&0&0&r^{-n}\\0&0&0&0&0&r^{-n}\\0&0&0&0&1&0\\0&0&0&1&0&0
\end{pmatrix}\in P_{\tH}(\Q_r)$$
we can check that $$Q^{-1}P^{-1}Q\cdot\i(m_2(w),A_nw) \in U_r^{\tH},$$ hence \begin{equation}\label{e:upsilonsteinformu2}
\Upsilon_r(Q\cdot\i(1,wA_nw),s)=\Upsilon_r(Q\cdot\i(m_2(w),A_nw),s)=r^{-6n(s+1/2)}\end{equation}

So, using~\eqref{e:whittakerramifiedsum},\eqref{e:whittakerramifiedsum2} \eqref{e:upsilonsteinformu} and~\eqref{e:upsilonsteinformu2}, \begin{align*}Z_r(1,1,s)&=\Upsilon_r(Q\cdot\i(1,1),s)
\int_{\Gamma_{0,r}^{\tF_1} }dh_r+\sum_{n >
0}\Upsilon_r(Q\cdot\i(1,A_n),s)\bigg[\int_{\Gamma_{0,r}^{\tF_1} A_n \Gamma_{0,r}^{\tF_1}}W_{\Psi,r}(h_r)dh_r\\& + \int_{\Gamma_{0,r}^{\tF_1} A_n \Gamma_{0,r}^{\tF_1}}W_{\Psi,r}(h_r)dh_r \bigg]\\ &=
[K_r^{\tF_1}:\Gamma_{0,r}^{\tF_1}]^{-1}\big(1 +2 \sum_{n >
0}\Upsilon_r(Q\cdot\i(1,A_n),s)\big)\\
&=\frac{1}{r+1}(1+2\sum_{n >
0}r^{-6n(s+1/2)})\\&=\frac{1}{r+1}\frac{1+r^{-6s-3}}{1-r^{-6s-3}} \end{align*}

whence~\eqref{e:aszrel} implies
\begin{equation}Z_r(g_r,1,s)=\frac{1}{r+1}W_{\Psi,r}(g_r) \cdot \frac{1+r^{-6s-3}}{1-r^{-6s-3}}.\end{equation}

Finally, we deal with the case when $k_r =s_1$ or $s_2$. The key observation is that if $k \in K_r^{\tF_1}$ then for $i=1,2$ $$s_i^{-1}m_2(k)s_i \in U^{\G}_r.$$ By the same argument as in~\eqref{e:doublecosetinvariance}, it follows that $\Upsilon_r(Q\cdot \i(s_i,g),s)$ only depends on the double
coset $ K_r^{\tF_1}g\Gamma_{0,r}^{\tF_1}.$ So, if we can show that for all $h\in \tF_1(\Q_r)$ we have $\sum_{a\in K_r^{\tF_1} h \Gamma_{0,r}^{\tF_1}/\Gamma_{0,r}^{\tF_1}}W_{\Psi,r}(g_ra) = 0$, it would follow that $Z_r(g_r,s_i,s)=0.$

If we define $$W(g_r) =\sum_{a\in K_r^{\tF_1} h \Gamma_{0,r}^{\tF_1}/\Gamma_{0,r}^{\tF_1}}W_{\Psi,r}(g_ra)$$ then $W(g_rk)=W(g_r)$ for all $k \in K_r^{\tF_1}$; in other words $W$ is a vector in the Whittaker space that is right $K_r^{\tF_1}$ invariant. But the only such vector is the $0$ vector and this completes the proof.

\end{proof}

\subsection{The local integral for primes in $S_2$} In this subsection, we prove the following proposition.\begin{proposition}\label{p:pullbacksteinbergsteinberg}Let $p$ be a prime dividing $\gcd(M,N)$ and $k_p \in Y_p$. We have

$Z_p(g_p, k_p, s)=\begin{cases} \frac{W_{\Psi,p}(g_p)}{(p+1)^2} &\text{ if } k_p =1 \text{ or }k_p=s_1 \\0 &\text{ otherwise }.\end{cases}$
\end{proposition}

\begin{proof} Recall that $\sigma$ is the irreducible automorphic representation of $GL_2(\A)$ generated by
$\Psi$. Let $\sigma_p$ be the local component of $\sigma$ at the place $p$. We know that $\sigma_p = \mathrm{Sp} \otimes \tau$ where $\mathrm{Sp}$ denotes the special (Steinberg) representation and
$\tau$ is a (possibly trivial) unramified quadratic character. We put $a_p = \tau(p)$, thus $a_p = \pm 1$ is the
eigenvalue of the local Hecke operator $T(p)$.

Let $\Gamma_{0,p}'^{\tF_1}$ denote the compact open subgroup of $\tF_1(\Q_p)$ defined by $$\Gamma_{0,p}'^{\tF_1}=\Gamma_{0,p}'^{\tF}\cap \tF_1(\Q_p).$$

We first consider the case $k_p =1$. By a similar argument as before,
 we have, \begin{equation}\label{e:doublecosetinvariance2}\begin{split}
\Upsilon_p(Q\cdot\i(1,k_1gk_2),s)&= \Upsilon_p(Q\cdot\i(m_2(k_1)m_2(k_1)^{-1},k_1gk_2),s)\\ &= \Upsilon_p(Q\cdot\i(m_2(k_1)^{-1},gk_2),s)\\
&=\Upsilon_p(Q\cdot\i(1,g),s)\end{split}\end{equation} In other words $\Upsilon_p(Q\cdot\i(1,g),s)$ only depends on the double coset $\Gamma_{0,p}'^{\tF_1}g\Gamma_{0,p}'^{\tF_1}.$

Because $p$ is inert in $L$, $L_p$ is a quadratic extension of $\Q_p$. We may write elements of $L_p$ in the form $a + b
\sqrt{-d}$ with $a,b \in \Q_p$; then $\Z_{L,p} = a +b \sqrt{-d}$ where $a,b \in \Z_p.$ Also note that $\Lambda_p$
is \emph{not} trivial.

Fix a set $U$ of representatives of $\Z_{L,p}^\times/\Gamma^0_{L,p}$. For definiteness we may take $$U =\{1\} \cup \{b + \sqrt{-d} : b \in\Z,0 \le b <p\}$$
For $l \in L_p^\times$ put $\widetilde{l} = \begin{pmatrix}l&0\\0&\overline{l}^{-1}\end{pmatrix}.$ We know that given $g \in \Gamma_{0,p}^{\tF_1}$ there exists $l \in \Z_{L,p}^\times$ such that $g\widetilde{l} \in\Gamma_{0,p}'^{\tF_1} .$ From this fact and the Bruhat-Cartan decomposition~\eqref{e:brucart}, it follows that

\begin{equation}\label{e:brucartsteinbsteinb}\begin{split}\tF_1(\Q_p)&= \bigsqcup_{l\in U}\Gamma_{0,p}'^{\tF_1}\widetilde{l}\Gamma_{0,p}'^{\tF_1}\quad \cup \quad \bigsqcup_{l\in U}\Gamma_{0,p}'^{\tF_1}w\widetilde{l} \Gamma_{0,p}'^{\tF_1}\\ &\cup \quad \bigsqcup_{\substack{n>0\\l\in U}} \Gamma_{0,p}'^{\tF_1}A_n\widetilde{l} \Gamma_{0,p}'^{\tF_1} \quad \cup \quad \bigsqcup_{\substack{n>0\\l\in U}} \Gamma_{0,p}'^{\tF_1}A_n w\widetilde{l}\Gamma_{0,p}'^{\tF_1}  \\ &\cup \quad \bigsqcup_{\substack{n>0\\l\in U}} \Gamma_{0,p}'^{\tF_1}wA_n\widetilde{l} \Gamma_{0,p}'^{\tF_1} \quad \cup\quad \bigsqcup_{\substack{n>0\\l\in U}} \Gamma_{0,p}'^{\tF_1}wA_n w\widetilde{l}\Gamma_{0,p}'^{\tF_1}.\end{split}\end{equation} where as before $A_n =
\begin{pmatrix}p^n&0\\0&p^{-n}\end{pmatrix}$ and $w=\begin{pmatrix}0&1\\-1&0\end{pmatrix}$.

Now, in the proof of Proposition~\ref{p:pullbackunramifiedsteinberg} we saw that the elements $Q\cdot\i(1,A_nw), Q\cdot\i(1,wA_n)$ of $\tH(\Q_p)$ do not belong to $P_{\tH}(\Q_p)QU_p^{\tH}$. In particular therefore, the elements $Q\cdot\i(1,A_nw\tl), Q\cdot\i(1,wA_n\tl)$ of $\tH(\Q_p)$ cannot belong to $P_{\tH}(\Q_p)QI_p'^{\tH}$.

So~\eqref{e:deflocalzeta} gives us \begin{equation}\label{e:eisensteinzetasteinbsteinb} \begin{split} Z_p(g_p,1,s) &= \sum_{l\in U}\Lambda_p^{-2}(l)\Upsilon_p(Q\cdot\i(1,\tl),s)
\int_{\tl \Gamma_{0,p}'^{\tF_1} }W_{\Psi,p}(g_ph_p)dh_p\\&+\sum_{n >
0}\sum_{l\in U}\Lambda_p^{-2}(l)\Upsilon_p(Q\cdot\i(1,A_n\tl),s)
\int_{\Gamma_{0,p}'^{\tF_1} A_n\tl \Gamma_{0,p}'^{\tF_1} }W_{\Psi,p}(g_ph_p)dh_p\\&+\sum_{n >
0}\sum_{l\in U}\Lambda_p^{-2}(l)\Upsilon_p(Q\cdot\i(1,wA_nw\tl),s)
\int_{\Gamma_{0,p}'^{\tF_1} wA_nw\tl \Gamma_{0,p}'^{\tF_1} }W_{\Psi,p}(g_ph_p)dh_p.\end{split}\end{equation}

If we choose $a_i,b_i$ as in the proof of Proposition~\ref{p:pullbackunramifiedsteinberg} then we have  $$\Gamma_{0,p}'^{\tF_1} A_n\tl \Gamma_{0,p}'^{\tF_1}=\bigsqcup_i a_i \tl \Gamma_{0,p}'^{\tF_1},$$ $$\Gamma_{0,p}'^{\tF_1} wA_nw\tl \Gamma_{0,p}'^{\tF_1}=\bigsqcup_i b_i \tl \Gamma_{0,p}'^{\tF_1}.$$ Hence, by the same argument as in the proof of that proposition, we have $$\int_{\Gamma_{0,p}'^{\tF_1} A_n\tl \Gamma_{0,p}'^{\tF_1} }W_{\Psi,p}(g_ph_p)dh_p=\int_{\Gamma_{0,p}'^{\tF_1} wA_nw\tl \Gamma_{0,p}'^{\tF_1} }W_{\Psi,p}(g_ph_p)dh_p= [K_p^{\tF} : \Gamma_{0,p}'^{\tF_1}]^{-1}.$$ It is easy to check that the last quantity is equal to $\frac{1}{(p+1)^2}$.

So we have \begin{equation}\label{e:eisensteinzetasteinbsteinb2}\begin{split} Z_p(g_p,1,s) &= \frac{W_{\Psi,p}(g_p)}{(p+1)^2}\bigg(\sum_{l\in U}\Lambda_p^{-2}(l)\Upsilon_p(Q\cdot\i(1,\tl),s)
\\&+\sum_{n >
0}\sum_{l\in U}\Lambda_p^{-2}(l)\Upsilon_p(Q\cdot\i(1,A_n\tl),s)+\sum_{n >
0}\sum_{l\in U}\Lambda_p^{-2}(l)\Upsilon_p(Q\cdot\i(1,wA_nw\tl),s)
 \bigg).\end{split}\end{equation}

We can check that for $n > 0$,  $Q\cdot\i(1,A_n\tl)$ does not belong to $P_{\tH}(\Q_p)QI_p'^{\tH}$, hence $\Upsilon_p(Q\cdot\i(1,A_n\tl),s)=0$. We can also check that for $l \ne 1, l \in U,$ $Q\cdot\i(1,\tl)$ does not belong to $P_{\tH}(\Q_p)QI_p'^{\tH}$, hence $\Upsilon_p(Q\cdot\i(1,\tl),s)=0$.

Also, putting $$P=\begin{pmatrix}0&0&p^n \overline{l}^{-1}&1&
0&0\\0&
1&
0&0&0&0\\
1&0&0&0&0&p^{-n}l\\0&0&0&0&0&p^{-n}l\\0&0&0&0&1&0\\0&0&0&1&0&0
\end{pmatrix}\in P_{\tH}(\Q_p)$$
we can check that $$Q^{-1}P^{-1}Q\cdot\i(w,A_nw\tl) \in I_p'^{\tH},$$ hence \begin{equation}\label{e:upsilonsteinformu3}
\Upsilon_p(Q\cdot\i(1,wA_nw\tl),s)=\Lambda_p(l)p^{-6n(s+1/2)}
\end{equation}

Thus we have $\Lambda_p^{-2}(l)\Upsilon_p(Q\cdot\i(1,wA_nw\tl),s) = \Lambda_p^{-1}(l)p^{-6n(s+1/2)}$ and hence for all $n >0 $ we have $$\sum_{l\in U}\Lambda_p^{-2}(l)\Upsilon_p(Q\cdot\i(1,wA_nw\tl),s)=0.$$

So we conclude that $$Z_p(g_p,1,s)=\frac{W_{\Psi,p}(g_p)}{(p+1)^2}.$$

Next, we deal with the case $k_p=s_1$.

If $k \in \Gamma_{0,p}'^{\tF_1}$ then $s_1^{-1}m_2(k)s_1 \in I_p'$. So, by the same argument as before, we know that $\Upsilon_p(Q\cdot\i(s_1,g),s)$ depends only on the double
coset $\Gamma_{0,p}'^{\tF_1}g\Gamma_{0,p}'^{\tF_1}.$

Also, by explicit computation, we check that $Q\cdot\i(s_1,A_nw\tl)$, $ Q\cdot\i(s_1,wA_n\tl)$ do not belong to $P_{\tH}(\Q_p)QI_p'^{\tH}$ for any $n \ge 0$. Moreover, the quantity $Q\cdot\i(s_1,A_n\tl)$ belongs to $P_{\tH}(\Q_p)QI_p'^{\tH}$ if and only if $n=0, l=1.$ On the other hand, for $n>0$, the quantity $Q\cdot\i(s_1, w A_nw\tl)$ does belong to $P_{\tH}(\Q_p)QI_p'^{\tH}$. By explicit computation which we omit, one sees that

\begin{equation}\label{e:eisensteinzetasteinbsteinbs1} Z_p(g_p,s_1,s) = \frac{W_{\Psi,p}(g_p)}{(p+1)^2}\bigg(1
+\sum_{n >
0}\sum_{l\in U}\Lambda_p^{-2}(l)\Upsilon_p(Q\cdot\i(s_1,wA_nw\tl),s)
 \bigg).\end{equation}

 But we check that  $\Upsilon_p(Q\cdot\i(s_1,wA_nw\tl),s)=\Lambda_p(l)p^{-6n(s+1/2)}$ and hence $\sum_{l\in U}\Lambda_p^{-2}(l)\Upsilon_p(Q\cdot\i(s_1,wA_nw\tl),s)=0$.

 This completes the proof that $$Z_p(g_p,s_1,s) = \frac{W_{\Psi,p}(g_p)}{(p+1)^2}.$$

 Next, we consider $k_p =s_2$. Let $\Gamma_p'^{0,\tF_1} = J_1\Gamma_{0,p}'^{\tF_1}J_1$ where $J_1 = \begin{pmatrix}0&1\\-1&0\end{pmatrix}.$

 If $k \in \Gamma_p'^{0,\tF_1}$ then $s_2^{-1}m_2(k)s_2 \in I_p'$. So, by the same argument as before, we know that $\Upsilon_p(Q\cdot\i(s_2,g),s)$ depends only on the double
coset $\Gamma_p'^{0,\tF_1}g\Gamma_{0,p}'^{\tF_1}.$

Now, the Bruhat-Cartan decomposition~\eqref{e:brucartsteinbsteinb} continues to hold when we replace the left $\Gamma_{0,p}'^{\tF_1}$ in each term by $\Gamma_p'^{0,\tF_1}.$ So, to prove that $Z_p(g_p,s_2,s) =0$ it is enough to prove that each of the elements $Q\cdot\i(s_2,A_n\tl)$, $Q\cdot\i(s_2,A_nw\tl)$, $Q\cdot\i(s_2,wA_n\tl)$, $Q\cdot\i(s_2,wA_nw\tl)$ cannot belong to $P_{\tH}(\Q_p)QI_p'^{\tH}$ for any $n\ge 0$. This we do by an explicit computation. The details are omitted.

Next, take $k_p = s_3$. Once again, we check that if $k \in \Gamma_p'^{0,\tF_1}$ then $s_3^{-1}m_2(k)s_3 \in I_p'$. On the other hand, an explicit computation again shows that the elements $Q\cdot\i(s_3,A_n\tl)$, $Q\cdot\i(s_3,A_nw\tl)$, $Q\cdot\i(s_3,wA_n\tl)$, $Q\cdot\i(s_3,wA_nw\tl)$ cannot belong to $P_{\tH}(\Q_p)QI_p'^{\tH}$. So by exactly the same argument as the previous case, $Z_p(g_p,s_3,s)=0.$

Next consider the case $k_p = \Theta$. Define $$\Gamma_{1,p}'^{\tF_1} =\{ A \in \Gamma_{0,p}'^{\tF_1} \mid A \equiv \begin{pmatrix}
1 & \ast\\
0 & 1
\end{pmatrix}\pmod{p} \}.$$ We can check that if $k \in \Gamma_{1,p}'^{\tF_1}$ then $\Theta^{-1}m_2(k)\Theta \in I_p'$. We know that given $g \in \Gamma_{0,p}^{\tF_1}$, there exists $l \in \Z_{L,p}^\times$ such that $g\widetilde{l} \in\Gamma_{1,p}'^{\tF_1} .$ Thus, the Bruhat-Cartan decomposition~\eqref{e:brucartsteinbsteinb} continues to hold when we replace the left $\Gamma_{0,p}'^{\tF_1}$ in each term by $\Gamma_{1,p}'^{\tF_1}.$ An explicit computation again shows that the elements $Q\cdot\i(\Theta,A_n\tl)$,
$Q\cdot\i(\Theta,A_nw\tl)$, $Q\cdot\i(\Theta,wA_n\tl)$ never belong to $P_{\tH}(\Q_p)QI_p'^{\tH}$. On the other hand, if $n>0$, then $Q\cdot\i(\Theta,wA_nw\tl)$ does belong to $P_{\tH}(\Q_p)QI_p'^{\tH}$. Indeed, by explicitly writing down the decomposition, we see that
\begin{equation}\label{e:eisensteinzetasteinbsteinbtheta} Z_p(g_p,\Theta,s) = \frac{W_{\Psi,p}(g_p)}{(p+1)^2}\bigg(\sum_{n >
0}\sum_{l\in U}\Lambda_p^{-2}(l)\Upsilon_p(Q\cdot\i(\Theta,wA_nw\tl),s)
 \bigg).\end{equation}

 But we see that  $\Upsilon_p(Q\cdot\i(\Theta,wA_nw\tl),s)=\Lambda_p(l)p^{-6n(s+1/2)}$ and hence $$\sum_{l\in U}\Lambda_p^{-2}(l)\Upsilon_p(Q\cdot\i(\Theta,wA_nw\tl),s)=0.$$

 This completes the proof that $$Z_p(g_p,\Theta,s) = 0.$$

The rest of the proof is similar: by explicit computations, we check that $Z_p(g_p,\Theta s_2,s) = 0$, $Z_p(g_p,\Theta s_4,s) = 0$, $Z_p(g_p,\Theta s_5,s) = 0$.

\end{proof}

\subsection{The local integral for primes in $S_1$} In this subsection, we prove the following proposition.\begin{proposition}\label{p:pullbacksteinbergunramified}Let $p$ be a prime dividing $M$ but not $N$ and $k_p \in Y_p$. We have

$Z_p(g_p, k_p, s)=\begin{cases} \frac{W_{\Psi,p}(g_p)}{(p+1)^2} &\text{ if } k_p =1 \text{ or } k_p=\Theta \\0 &\text{ otherwise }.\end{cases}$
\end{proposition}

\begin{proof} Recall that $\sigma$ is the irreducible automorphic representation of $GL_2(\A)$ generated by
$\Psi$. Let $\sigma_p$ be the local component of $\sigma$ at the place $p$. We also let $\alpha,
\beta$ be the unramified characters of $\Q_p^\times$ from which $\sigma_p$ is induced.

Let $\Gamma_{0,p}'^{\tF_1} ,\Gamma_{1,p}'^{\tF_1} $ be as defined in the previous subsection.

We first consider the case $k_p =1$.
As in the previous case, $\Upsilon_p(Q\cdot\i(1,g),s)$ only depends on the double coset $\Gamma_{0,p}'^{\tF_1}g\Gamma_{0,p}'^{\tF_1}.$

By explicit computation we check that, $Q\cdot\i(1,A_n\tl)$, $Q\cdot\i(1,A_nw\tl)$, $Q\cdot\i(1,wA_n\tl)$, $Q\cdot\i(1,wA_nw\tl)$ do not belong to $P_{\tH}(\Q_p)\Omega I_p'^{\tH}$. Thus only the section supported on $Q$ contributes. So, by the results of the previous subsection, and by~\eqref{e:brucartsteinbsteinb}, we have $Z_p(g_p,1,s)=\frac{W_{\Psi,p}(g_p)}{(p+1)^2}$.

Next, consider the case $k_p=s_1$. Again, by explicit computation, we check that for $n > 0$,  $Q\cdot\i(s_1,A_n\tl)$, $Q\cdot\i(s_1,A_nw\tl)$, $Q\cdot\i(s_1,wA_n\tl)$, $Q\cdot\i(s_1,wA_nw\tl)$ do not belong to $P_{\tH}(\Q_p)\Omega I_p'^{\tH}$. Furthermore $Q\cdot\i(s_1, w\tl)$ does not belong to $P_{\tH}(\Q_p)\Omega I_p'^{\tH}$ and $Q\cdot\i(s_1,\tl)$ belongs only when $l \ne 1$. So

\begin{equation}\label{e:eisensteinzetasteinbunrams1} Z_p(g_p,s_1,s) = \frac{W_{\Psi,p}(g_p)}{(p+1)^2}\left(\sum_{\substack{l\in U\\ l\ne 1}}\Lambda_p^{-2}(l)\Upsilon_p(Q\cdot\i(s_1,\tl),s) +1\right)
 .\end{equation} where the 1 comes from the results of the previous subsection.

 Noting that $\Upsilon_p(Q\cdot\i(s_1,\tl),s)= \Lambda_p(l)$ and that $\sum_{\substack{l\in U\\ l\ne 1}} \Lambda_p^{-1}(l) = -1$,

 we get $$Z_p(g_p,s_1,s) = -\frac{W_{\Psi,p}(g_p)}{(p+1)^2}+\frac{W_{\Psi,p}(g_p)}{(p+1)^2}=0.$$

 Next, we consider $k_p =s_2$. Let $\Gamma_p'^{0,\tF_1}$ be as in the previous subsection.

 By the argument there, we know that $\Upsilon_p(Q\cdot\i(s_2,g),s)$ depends only on the double
coset $\Gamma_p'^{0,\tF_1}g\Gamma_{0,p}'^{\tF_1}.$

To prove that $Z_p(g_p,s_2,s) =0$ it is enough to prove that each of the elements $Q\cdot\i(s_2,A_n\tl)$, $Q\cdot\i(s_2,A_nw\tl)$, $Q\cdot\i(s_2,wA_n\tl)$, $Q\cdot\i(s_2,wA_nw\tl)$ cannot belong to $P_{\tH}(\Q_p)\Omega I_p'^{\tH}$ for any $n\ge 0$. This we do by an explicit computation. The details are omitted.

Next, take $k_p = s_3$. Once again, an explicit computation shows that the elements $Q\cdot\i(s_3,A_n\tl)$, $Q\cdot\i(s_3,A_nw\tl)$, $ Q\cdot\i(s_3,wA_n\tl)$, $Q\cdot\i(s_3,wA_nw\tl)$ cannot belong to $P_{\tH}(\Q_p)\Omega I_p'^{\tH}$. So by exactly the same argument as the previous case, $Z_p(g_p,s_3,s)=0.$

Next, consider the case $k_p = \Theta$. By explicit calculation, we check that for $n>0$ the elements $Q\cdot\i(\Theta,A_nw\tl)$, $Q\cdot\i(\Theta,wA_n\tl)$, $Q\cdot\i(\Theta,wA_nw\tl)$ do not belong to $P_{\tH}(\Q_p)\Omega I_p'^{\tH}$. Also check that $Q\cdot\i(\Theta, w\tl) \notin P_{\tH}(\Q_p)\Omega I_p'^{\tH}$. Also, provided $l \ne 1$, we have $Q\cdot\i(\Theta, w\tl) \notin P_{\tH}(\Q_p)\Omega I_p'^{\tH}. $ Thus, the only term that contributes is $Q\cdot\i(\Theta, 1)$.

So by the same argument as before, we have \begin{equation}\label{e:eisensteinzetasteinbunram} \begin{split}Z_p(g_p,\Theta,s) &= \Upsilon_p(Q\cdot\i(\Theta,1),s)
\int_{\Gamma_{0,p}'^{\tF_1} }W_{\Psi,p}(g_ph_p)dh_p\\&= \frac{W_{\Psi,p}(g_p)}{(p+1)^2}.\end{split}\end{equation}

Next consider the case $k_p = \Theta s_2.$ For $x\in \Z$, let $u(x)$ be as in the previous subsection. As before, to prove that $Z_p(g_p,\Theta s_2,s) = 0$, it is enough to check that the elements $Q\cdot\i(\Theta s_2,u(x)A_n\tl)$, $Q\cdot\i(\Theta s_2,u(x)A_nw\tl)$, $Q\cdot\i(\Theta s_2,u(x) wA_n\tl)$, $Q\cdot\i(\Theta,u(x) wA_nw\tl)$ do not belong to $P_{\tH}(\Q_p)\Omega I_p'^{\tH}$. This can be done by an explicit computation (omitted for brevity).

Next, consider the case $k_p = \Theta s_4$. To prove that $Z_p(g_p,\Theta s_4,s) = 0$, it is enough to check that the elements $Q\cdot\i(\Theta s_4,u(x)A_n\tl)$, $Q\cdot\i(\Theta s_4,u(x)A_nw\tl)$, $Q\cdot\i(\Theta s_4,u(x) wA_n\tl)$, $Q\cdot\i(\Theta s_4,u(x) wA_nw\tl)$ do not belong to $P_{\tH}(\Q_p)\Omega I_p'^{\tH}$. This is done by an explicit computation, which we omit.

Finally, we consider the case $k_p = \Theta s_5$. To prove that $Z_p(g_p,\Theta s_5,s) = 0$, it is enough to check that the elements $Q\cdot\i(\Theta s_5,A_n\tl)$, $Q\cdot\i(\Theta s_5,A_nw\tl)$, $Q\cdot\i(\Theta s_5,wA_n\tl)$, $Q\cdot\i(\Theta s_5, wA_nw\tl)$ do not belong to $P_{\tH}(\Q_p)\Omega I_p'^{\tH}$. This is done by an explicit computation, which we omit.

\end{proof}

\subsection{The local integral at infinity} In this subsection we prove the following proposition.

\begin{proposition}We have $$Z_\infty(g_\infty,1,s)=B_\infty(s)W_{\Psi, \infty}(g_\infty),$$ where $B_\infty(s)=\frac{ (-1)^{\ell/2}2^{-6s -1} \pi }{6s+\ell -1}.$
\end{proposition}

\begin{proof}
Note that $K_\infty^{\tF}$ is the maximal compact subgroup of $\tF_1(\R)$. Furthermore, note that any element $h$ of $\tF_1(\R)$ can be written in the form $$h=\begin{pmatrix}1&x\\0&1\end{pmatrix}
\begin{pmatrix}b&0\\0&b^{-1}\end{pmatrix}k$$ where $x\in \R$, $b\in \R^+$, $k\in K_\infty^{\tF}$. Let us henceforth denote $u(x)=\begin{pmatrix}1&x\\0&1\end{pmatrix}$, $t(b)=\begin{pmatrix}b&0\\0&b^{-1}\end{pmatrix}$. We normalize our Haar measures such that $K_\infty^{\tF}$ has volume 1. Also, note that $\Lambda_\infty$ is trivial and for $k\in K_\infty^{\tF}$, $g,h \in  \tF_1(\R)$ we have $$\Upsilon_\infty(Q \cdot\i(1,hk),s)W_{\Psi,\infty}(ghk)
=\Upsilon_\infty(Q \cdot\i(1,h),s)W_{\Psi,\infty}(gh).$$

Hence we have
\begin{equation}Z_\infty(g_\infty,1,s)=\int_0^\infty
\int_{-\infty}^{\infty}\Upsilon_\infty(Q \cdot\i(1,u(x)t(b)),s)
W_{\Psi,\infty}(g_\infty u(x)t(b))b^{-3}dxdb\end{equation} where $dx, db$ are the usual Lebesgue measures.

Let $K_\infty^H =K_\infty^{\tH} \cap H(\R)$. To calculate $\Upsilon_\infty(Q \cdot\i(1,u(x)t(b)),s)$ we need to write the Iwasawa decomposition of $Q\cdot\i(1,u(x)t(b)).$ However, finding an explicit decomposition is not really necessary. Indeed, we know that there exists some decomposition \begin{equation}\label{iwasawainf}Q\cdot\i(1,u(x)t(b))=
\begin{pmatrix}A&X\\0&(A^t)^{-1}\end{pmatrix}K$$ with $K \in K_\infty^H$, $A\in GL_3(\R)$ and that
$$\Upsilon_\infty(Q \cdot\i(1,u(x)t(b)),s) = |\det(A)|^{6(s+1/2)}\det(J(K, i))^{-\ell}.\end{equation}

Now, let $A_x^b = Q\cdot\i(1,u(x)t(b))$. By explicit computation, we see that $$A_x^b = \begin{pmatrix}0&1&0&0&0&0\\1&0&0&0&0&0\\0&0&0&0&0&-\frac{1}{b}
\\0&0&0&0&1&-\frac{1}{b}\\0&0&0&1&0&0\\0&1&b&0&0&-\frac{x}{b}\end{pmatrix} $$

By~\eqref{iwasawainf} we have $$\det(J(A_x^b, i)) = \det(A)^{-1} \det(J(K,i)).$$ Since $\det(J(A_x^b, i)) = \frac{x-i(b^2+1)}{b}$ we have \begin{equation}\Upsilon_\infty(Q \cdot\i(1,u(x)t(b)),s) = |\det(A)|^{6(s+1/2)}\det(A)^{-\ell}b^\ell (x-i(b^2+1))^{-\ell}.\end{equation}

On the other hand, we have $$(A_x^b)(i)= (AA^t i + XA^t).$$ By explicit computation, we see that \begin{align*}(A_x^b)(i)&= \frac{1}{b^4 +2b^2 + x^2 +1}\bigg[
\begin{pmatrix}b^4+b^2+x^2&0&-x\\0&b^4+2b^2+x^2+1&0\\-x&0&b^2+1
\end{pmatrix}i \\& \quad + \begin{pmatrix}-x&0&b^2+1\\0&0&0\\b^2+1&0&x\end{pmatrix}
\bigg]\end{align*}

From this we get $\det(A)= \frac{b}{\sqrt{b^4+1+2b^2+x^2}}$.

Therefore, we have \begin{equation}\Upsilon_\infty(Q \cdot\i(1,u(x)t(b)),s) = b^{6s+3}(b^4+1+2b^2+x^2)^{-3(s+1/2)+\ell/2}(x-i(b^2+1))^{-\ell}.\end{equation}

On the other hand, we know that the normalized Whittaker function satisfies $$W_{\Psi,\infty}(u(x)t(b))=e^{2\pi i x}e^{-2\pi b^2}b^\ell$$

We will prove the proposition only for $g_\infty=1$, the calculations in the general case are similar.

We need to evaluate the integral \begin{equation}\int_0^\infty
\int_{-\infty}^{\infty}b^{6s+\ell}(x-i(b^2+1))^{-3(s+1/2)-\ell/2}
(x+i(b^2+1))^{-3(s+1/2)+\ell/2}
e^{-2\pi i x}e^{2\pi b^2}dxdb\end{equation}

Putting $b^2=y$ , the above integral becomes \begin{equation}\label{infint2}\frac{1}{2}\int_0^\infty
\int_{-\infty}^{\infty}y^{3s+\frac{\ell-1}{2}}(x-i(y+1))
^{-3(s+\frac{1}{2})-\frac{\ell}{2}}
(x+i(y+1))^{-3(s+\frac{1}{2})+\frac{\ell}{2}}
e^{2\pi i x}e^{-2\pi y}dxdy\end{equation}

Applying~\cite[(6.11)]{grokud} to the inner integral,~\eqref{infint2} becomes $$\frac{(-1)^{\ell/2} (2\pi)^{6s+3}}{2\Gamma(3s+\frac{3}{2}+\frac{\ell}{2})
\Gamma(3s+\frac{3}{2}-\frac{\ell}{2})}$$ times
\begin{equation}\label{infint3}\int_0^\infty e^{-2\pi (1+2t)}(t+1)
^{3s+\frac{1}{2}+\frac{\ell}{2}}
t^{3s+\frac{1}{2}-\frac{\ell}{2}}
\left(\int_{0}^{\infty}y^{3s+\frac{\ell-1}{2}}
e^{-4\pi y(1+t)}dy\right)dt.\end{equation}

Now, $\int_{0}^{\infty}y^{3s+\frac{\ell-1}{2}}
e^{-4\pi y(1+t)}dy$ evaluates to $$2^{-6s-\ell-1}\pi^{-3s-\frac{\ell}{2}-\frac{1}{2}}
\Gamma(3s+\ell/2+\frac{1}{2}).$$ Using this, and the formula $$\int_0^\infty e^{-2\pi (1+2t)}
t^{3s+\frac{1}{2}-\frac{\ell}{2}}dt= 2^{-6s+\ell -3} e^{-2\pi}\pi^{-3s+\ell/2-\frac{3}{2}}
\Gamma(3s+\frac{3}{2}-\frac{\ell}{2})$$

we see that~\eqref{infint2} simplifies to $$\frac{ (-1)^{l/2}2^{-6s -1} \pi }{6s+\ell -1} W_{\Psi,\infty}(1).$$

\end{proof}

\section{Proof of the Pullback formula}\label{s:pullback}

In this section, we will prove Theorem~\ref{t:pullback}.

Recall the definition of $\mathcal{E}(g,s)$ from Subsection~\ref{s:defpullback}. Our main step in computing $\mathcal{E}(g,s)$ will be the evaluation of the following integral:
\begin{equation}\Upsilon_\Psi(g,s) = \int_{ Q \tF_1[g](\A)}\Upsilon(\i(g,h),s)\Psi(h)\Lambda^{-1}(\det h)dh\end{equation}

 By~\cite{shibook1}, we know that the integral above converges absolutely and uniformly on compact sets for $\Re(s)$ large. We are going to evaluate the above integral for such $s$.

Note that $\G(\A) = P(\A)\prod_vK_v^{\G}.$ Moreover if $k \in K_v^{\G}$, we may write
$$k =
m_2\left(\begin{pmatrix}\lambda&0\\0&1\end{pmatrix}\right) k'$$ where $\lambda = \mu_2(k)$, so that $\mu_2(k')=1$.

For any $p \in S_3$ we have, by the Bruhat decomposition, $$K_p^{\G}= (P(\Q_p)\cap K_p^{\G})U_p^{\G} \sqcup (P(\Q_p)\cap K_p^{\G})s_1U_p^{\G} \sqcup (P(\Q_p)\cap K_p^{\G})s_2U_p^{\G}.$$

Also, for $p |M$, we have, by Lemma~\ref{sirepresentatives},
$$K_p^{\G}=  \coprod_{s \in Y_p} (P(\Q_p)\cap K_p^{\G})sI_p'.$$

Recall that we defined the compact subgroup $U^{\G}$ of $\G(\A_f)$ in~\eqref{def:kgtilde}.

So write $g=m_1(a)m_2(b)nk$ where $k\in \prod_v K_v^{\G} $,
$\mu_2(k)=1$ and further write $k=k_\infty k_{\text{ram}} k_{\text{ur}}$ where $$k_\infty \in K_{\infty}^{\G}, \ k_{\text{ur}} \in U^{\G}$$

 and  $k_{\text{ram}}=\prod_v(k_{\text{ram}})_v$, with
$$ (k_{\text{ram}})_v \in \begin{cases}\ \{1\}&\text{ if } v \notin S \\ \{1,s_1,s_2\} &\text{ if } v \in S_3\\ \{1,s_1,s_2,s_3,\Theta,\Theta s_2, \Theta s_4, \Theta s_5 \} &\text{ if } v \in S_1 \sqcup S_2.\end{cases}$$

Therefore we have \begin{align*}\Upsilon_\Psi(g,s)&=\int_{
\tF_1[m_2(b)](\A)}\Upsilon(Q \cdot\i(m_1(a)m_2(b)nk,b(b^{-1}h)),s) \Psi(h)\Lambda^{-1}(\det h)dh
\\&=\rho_\ell(k_\infty)\\&\times\int_{   \tF_1[m_2(b)](\A)}\Upsilon(Q \cdot\i(m_1(a)m_2(b)nk_{\text{ram}},b(b^{-1}h)),s)
\Psi(h)\Lambda^{-1}(\det h)dh \\ & \qquad \qquad(\text{using
properties from Subsection~\ref{s:sectionproperties})}\\&=\Lambda(a)|N_{L/\Q}(a)\cdot \mu_2(b)^{-1}|^{3(s +
1/2)}\rho_\ell(k_\infty)\\&\times \int_{  \tF_1(\A)}\Upsilon(Q \cdot\i(k_{\text{ram}},h),s)
\Psi(bh)\Lambda^{-1}(\det h)dh \\&\qquad \qquad(\text{using}~\eqref{e:upsilondefformula} ). \end{align*}

We write $$U_b(k_{\text{ram}},s) = \int_{  \tF_1(\A)}\Upsilon(Q \cdot\i(k_{\text{ram}},h),s) \Psi(bh)\Lambda^{-1}(\det h)dh.
$$

Thus we have \begin{equation}\label{e:upsilonsimplified}\Upsilon_\Psi(g,s)
=\Lambda(a)|N_{L/\Q}(a)\cdot
\mu_2(b)^{-1}|^{3(s + 1/2)}\rho_l(k_\infty)\times U_b(k_{\text{ram}},s)\end{equation}

Recall the Whittaker expansion
\begin{equation}\label{defwhitt}\Psi(g) = \sum_{\lambda \in \Q^\times}W_{\Psi}\left(\begin{pmatrix}\lambda&0\\0&1
\end{pmatrix}g\right)\end{equation}

Therefore \begin{equation}\label{defub}U_b(k_{\text{ram}},s)=\sum_{\lambda \in
\Q^\times}Z\left(\begin{pmatrix}\lambda&0\\0&1\end{pmatrix}b,k_{\text{ram}},s\right)\end{equation} where for $g \in \tF(\A)$, $k \in \prod_vK_v^{\G}, \mu_2(k)=1$, we define
$$Z(g,k,s) = \int_{  \tF_1(\A)}\Upsilon(Q \cdot\i(k,h),s) W_{\Psi}(gh)\Lambda^{-1}(\det h)dh.$$

Note that the uniqueness of the Whittaker function implies $$Z(g,k,s) = \prod_v Z_v(g_v, k_v, s),$$ where the local zeta integral $Z_v(g_v, k_v, s)$ is defined as in~\eqref{e:deflocalzeta}.

So, by the results of the previous two sections, we have \begin{equation}\label{e:zgks}Z(g,k_{\text{ram}},s)= \begin{cases} B(s)W_{\Psi}(g) & \text{if } (k_{ \text{ram}})_v \in Y_v' \text{ for all places } v \\ 0 & \text{otherwise}\end{cases}\end{equation} where we define $$Y_v'= \begin{cases}\ \{1\}&\text{ if } v \notin S_1 \sqcup S_2 \\ \{1,s_1\} &\text{ if } v \in S_2\\ \{1,\Theta \} &\text{ if } v \in S_1.\end{cases}$$

From~\eqref{e:upsilonsimplified},\eqref{defwhitt},\eqref{defub},\eqref{e:zgks} we conclude that

\begin{equation}\Upsilon_\Psi(g,s) = B(s)f_{\Lambda}(g,s)\end{equation}

where $f_{\Lambda}(g,s)$ is defined as in Section~\ref{s:eisensteindef}.

We are now in a position to prove the Pullback formula.

\begin{proof}[Proof of Theorem~\ref{t:pullback}]
Recall the definition of $B(s)$ from~\eqref{defB(s)}. Also recall that we defined \begin{equation}\mathcal{E}(g,s)=\int_{  \widetilde{F_1}(\Q) \backslash
\tF_1[g](\A)}E_{\Upsilon}(\i(g,h),s)\Psi(h)\Lambda^{-1}(\det h) dh.\end{equation} The pullback formula states that $$\mathcal{E}(g,s)=B(s)E_{\Psi,\Lambda}(g,s).$$

Since $E_{\Upsilon}$ is left invariant by $\tH(\Q)$, we have 
\begin{equation}\mathcal{E}(g,s)=\int_{  \widetilde{F_1}(\Q) \backslash
\tF_1[g](\A)}E_{\Upsilon}(Q \cdot\i(g,h),s)\Psi(h)\Lambda^{-1}(\det h) dh.\end{equation}

By abuse of notation, we use $\tR(\Q)$ to denote its image in $\tH(\Q)$. Let $V(\Q)=Q \tR(\Q) Q^{-1}$. First, we recall from~\cite{shibook1} that $|P_{\tH}(\Q)\bs \tH(\Q) / V(\Q)|$=2. We take the identity element as one of the double coset representatives, and denote the other one by $\tau$.
Thus $$\tH(\Q)= P_{\tH}(\Q)V(\Q) \sqcup P_{\tH}(\Q)\tau V(\Q).$$

Let us denote by $R_1,R_2$ the corresponding sets of coset representatives, i.e. $R_1 \subset V(\Q), R_2 \subset \tau V(\Q)$ and

$$P_{\tH}(\Q)V(\Q) = \bigsqcup_{s \in R_1}P_{\tH}(\Q)s$$ and $$P_{\tH}(\Q)\tau V(\Q) = \bigsqcup_{s \in R_2}P_{\tH}(\Q)s.$$
Recall that we defined $$E_{\Upsilon}(h,s) = \sum_{\gamma \in P_{\tH}(\Q) \bs \tH(\Q)} \Upsilon(\gamma h,s)$$ for $\Re(s)$ large. We can write $E_{\Upsilon}(h,s) =E_{\Upsilon}^1(h,s) +  E_{\Upsilon}^2(h,s)$ where
$$E_{\Upsilon}^1(h,s)= \sum_{\gamma \in R_1} \Upsilon(\gamma h,s)$$ and $$E_{\Upsilon}^2(h,s)= \sum_{\gamma \in R_2} \Upsilon(\gamma h,s).$$

Now, by~\cite[22.9]{shibook1} the orbit of $\tau$ is 'negligible' for our integral, that is for all $g$,
$$\int_{  \widetilde{F_1}(\Q) \backslash
\tF_1[g](\A)}E_{\Upsilon}^2(Q\cdot\i(g,h),s)\Psi(h)\Lambda^{-1}(\det h) dh =0.$$

It follows that

\begin{equation}\label{negligibleeisenstein}\mathcal{E}(g,s)=\int_{  \widetilde{F_1}(\Q) \backslash
\tF_1[g](\A)}E_{\Upsilon}^1(Q \i(g,h),s)\Psi(h)\Lambda^{-1}(\det h) dh.\end{equation}

On the other hand, by~\cite[2.7]{shibook1} we can take $R_1$ to be the following set:

\begin{equation}
R_1= \{Q\cdot\i(m_2(\xi)\beta, 1)Q^{-1} : \xi \in \tF_1(\Q), \beta \in P(\Q) \bs \G(\Q)\} \end{equation}

For $\Re(s)$ large, we therefore have $$E_{\Upsilon}^1(Q\cdot\i(g,h),s) = \sum_{\substack{\xi \in \tF_1(\Q)\\ \beta \in P(\Q) \bs \G(\Q)}}\Upsilon(Q\cdot\i((m_2(\xi)\beta g,h),s).$$ Substituting in~\eqref{negligibleeisenstein} we have \begin{align*}\mathcal{E}(g,s)&=\int_{  \widetilde{F_1}(\Q) \backslash
\tF_1[g](\A)}\sum_{\substack{\xi \in \tF_1(\Q)\\ \beta \in P(\Q) \bs \G(\Q)}}\Upsilon(Q\cdot\i(m_2(\xi)\beta g,h),s)\Psi(h)\Lambda^{-1}(\det h) dh\\&=\int_{  \widetilde{F_1}(\Q) \backslash
\tF_1[g](\A)}\sum_{\substack{\xi \in \tF_1(\Q)\\ \beta \in P(\Q) \bs \G(\Q)}}\Upsilon(Q\cdot\i(\beta g,\xi ^{-1}h),s)\Psi(\xi^{-1}h)\Lambda^{-1}(\det \xi^{-1}h) dh\\&=\sum_{\beta \in P(\Q) \bs \G(\Q)}\int_{\tF_1[g](\A)}\Upsilon(Q\cdot\i(\beta g,h),s)\Psi(h)\Lambda^{-1}(\det h) dh\\&=\sum_{\beta \in P(\Q) \bs \G(\Q)}\Upsilon_\Psi(\beta g,s)\\&=B(s)\sum_{\beta \in P(\Q) \bs \G(\Q)}f_{\Lambda}(\beta g,s)\\&=B(s)E_{\Psi, \Lambda}(g,s)
\end{align*}

Thus \begin{equation}\label{pullcrucial}\int_{  \widetilde{F_1}(\Q) \backslash
\tF_1[g](\A)}E_{\Upsilon}(\i(g,h),s)\widetilde{\Psi}(h)\Lambda^{-1}(\det h) dh = B(s)E_{\Psi, \Lambda}(g,s)\end{equation} for $\Re(s)$ large (so that all sums and integrals converge nicely and our manipulations are valid).

However, $E_{\Upsilon}(\i(g,h),s)$ is slowly increasing away from its poles, while $\Psi(h)$ is rapidly decreasing. Thus the left side above converges absolutely for $s \in \C$ away from the poles of the Eisenstein series. Hence~\eqref{pullcrucial} holds as an identity of meromorphic functions.

\end{proof}

\section{Integral representations for holomorphic forms}\label{s:intrep}
\subsection{Siegel newforms of squarefree level}
For $M$ a positive integer define the following global congruence subgroups.

\begin{align*}B(M) &:= Sp(4,\Z) \cap \begin{pmatrix}\Z& M\Z&\Z&\Z\\\Z& \Z&\Z&\Z\\M\Z& M\Z&\Z&\Z\\M\Z&M \Z&M\Z&\Z\\\end{pmatrix},\\
U_1(M) &:= Sp(4,\Z) \cap \begin{pmatrix}\Z& \Z&\Z&\Z\\\Z& \Z&\Z&\Z\\M\Z& M\Z&\Z&\Z\\M\Z&M \Z&\Z&\Z\\\end{pmatrix},\\
U_2(M) &:= Sp(4,\Z) \cap \begin{pmatrix}\Z& M\Z&\Z&\Z\\\Z& \Z&\Z&\Z\\\Z& M\Z&\Z&\Z\\M\Z&M \Z&M\Z&\Z\\\end{pmatrix},\\
U_0(M) &:= Sp(4,\Q) \cap \begin{pmatrix}\Z& M\Z&\Z&\Z\\\Z& \Z&\Z&M^{-1}\Z\\M\Z& M\Z&\Z&\Z\\M\Z&M
\Z&M\Z&\Z\\\end{pmatrix}.\\ \end{align*}

When $M=1$ each of the above groups is simply $Sp(4,\Z)$. For $M>1$, the groups are all distinct. If $\Gamma'$ is equal to one of the above groups, or (more generally) is any congruence subgroup, we define $S_k(\Gamma')$ to be
the space of Siegel cusp forms of degree 2 and weight $k$ with respect to the group $\Gamma'$.

More precisely, let $\H_2= \{ Z \in M_2(\C) | Z =Z^t,i( \overline{Z} - Z)$ is positive definite$\}$. For any $g=
\begin{pmatrix} A&B\\ C&D \end{pmatrix} \in G$ let $J(g,Z) = CZ + D$. Then $f \in S_k(\Gamma')$ if it is a holomorphic
function on $\H_2$, satisfies $f(\gamma Z) = \det(J(\gamma,Z))^k f(Z)$ for $\gamma \in \Gamma',Z \in \H_2$ and
vanishes at the cusps. It is well-known that $f$ has a Fourier expansion $$f(Z) =
\sum_{S > 0} a(S, F) e(\text{tr}(SZ)),$$ where $e(z) = \exp(2\pi iz)$ and $S$ runs through all symmetric
semi-integral positive-definite matrices of size two.

Now let $M$ be a square-free positive integer. For any decomposition $M= M_1M_2$ into coprime integers we
define, following Schmidt~\cite{sch}, the subspace of oldforms $S_k(B(M))^{\text{old}}$ to be the sum of the
spaces

$$S_k(B(M_1) \cap U_0(M_2)) + S_k(B(M_1) \cap U_1(M_2)) + S_k(B(M_1) \cap U_2(M_2)).$$

For each prime $p$ not dividing $M$ there is the local Hecke algebra $\mathfrak{H}_p$ of operators on
$S_k(B(M))$ and for each prime $q$ dividing $M$ we have the Atkin-Lehner involution $\eta_q$ also acting on
$S_k(B(M))$. For details, the reader may refer to \cite{sch}.

By a newform for the minimal congruence subgroup $B(M)$, we mean an element $f \in  S_k(B(M))$ with the
following properties
\begin{enumerate}
\item \label{itemfa} $f$ lies in the orthogonal complement of the space $S_k(B(M))^{\text{old}}.$
\item \label{itemfb}$f$ is an eigenform for the local Hecke algebras $\mathfrak{H}_p$ for all primes $p$ not dividing $M$.
\item \label{itemfc}$f$ is an eigenform for the Atkin-Lehner involutions $\eta_q$ for all primes $q$ dividing $M$.
\end{enumerate}

\textbf{Remark.} By \cite{sch}, if we assume the hypothesis that a nice $L$-function theory for
$GSp(4)$ exists, \eqref{itemfb} and \eqref{itemfc} above follow from \eqref{itemfa} and the assumption that $f$
is an eigenform for the local Hecke algebras at \emph{almost} all primes.

\subsection{Description of $F$ and $\Lambda$}

Let $M$ be an odd square-free positive integer and $$F(Z) = \sum_{T>0} a(T)\text{e}(\text{tr}(TZ)) $$ be a Siegel newform for $B(M)$ of even weight $\ell$.

We make the following assumption:

\begin{equation}\label{fundamentalrestriction}a(T) \neq 0
\text{ \emph{for some} } T
= \begin{pmatrix} a& \frac{b}{2} \\ \frac{b}{2} & c \end{pmatrix}\end{equation} \emph{such that} $-d = b^2
-4ac$  \emph{is the discriminant of the imaginary quadratic field} $\Q(\sqrt{-d}),$
 \emph{and all primes dividing} $MN$ \emph{are inert in}  $\Q(\sqrt{-d}).$

We define a function $\Phi = \Phi_{F}$ on $G(\A)$ by
$$\Phi(\gamma g_\infty k_0) = \mu_2(g_\infty)^l\det(J(g_\infty, iI_2))^{-l}F(g_\infty(i))$$
where $\gamma \in G(\Q), g_\infty \in G(\R)^+$ and $$k_0 \in (\prod_{ p \nmid M} K_p) \cdot (\prod_{p \mid M}I_p). $$

Because we do not have strong multiplicity one for $G$ we can only say that the representation of $G(\A)$
generated by $\Phi$ is a \emph{multiple} of an irreducible representation $\pi$. However that is enough for our
purposes.

We know that $\pi =\otimes \pi_v$ where $$\pi_v =  \begin{cases}\text{holomorphic discrete series} & \text{ if } v=\infty,\\
 \text{unramified spherical principal series} &\text{ if } v \text{ finite }, v \nmid M, \\
 \xi_v \text{St}_{GSp(4)}  \text{where } \xi_v \text{ unramified, } \xi_v^2=1 &\text{ if } v \mid M. \end{cases}$$

Put $L=\Q(\sqrt{-d}).$ where $d$ is the integer defined in~\eqref{fundamentalrestriction}. Thus, we have fixed a choice for the imaginary quadratic field $L$, which was till now assumed to be more or less arbitrary.

Next we will fix a choice for $\Lambda$. The choice, like that of $L$ will depend on $F$.  Basically $\Lambda$ is a Hecke character satisfying the four assumptions of Section~\ref{section1} such that $F$ has a non-trivial Bessel model for $\Lambda$. More precisely, we choose the character $\Lambda$ and define the quantity $a(\Lambda)$ as in~\cite[Subsection 8.3]{lfshort}.

\subsection{The integral representation}
The following theorem was proved in~\cite{lfshort}.

\begin{theorem}[\cite{lfshort}, Theorem~8.5.1] $$\int_{Z_G(\A)G(\Q)\backslash G(\A)}E_{\Psi,\Lambda}(g,s)\overline{\Phi}(g)dg=C(s)\cdot L(3s + \frac{1}{2}, F \times  g)$$ where $C(s)=$ $$\frac{Q_f\pi \overline{a(\Lambda)}(4\pi)^{-3s - \frac{3}{2}\ell +
\frac{3}{2}}d^{-3s-\frac{\ell}{2}}\Gamma(3s +  \frac{3}{2}l - \frac{3}{2})}{\sigma_1(M/f)P_{MN}(6s +\ell -1)\zeta^{MN}(6s+1)L(3s+1,\sigma \times \rho(\Lambda))} \prod_{p |f}
\frac{p^{-6s-3}}{1-a_{p}w_{p}p^{-3s-3/2}}$$ where $$f = \gcd(M, N),$$

$$Q_{A} = \prod_{\substack{r \mid A\\ r \text{ prime}}} (1-r),$$

and $\sigma_1(A), P_A,\zeta^A$ are as defined earlier.
\end{theorem}
\textbf{Remark}. For related results, see~\cite{fur}, \cite{pitsch2}, \cite{pitsch}.

Recall the definition of $B(s)$ from~\eqref{defB(s)} and let $$A(s)=B(s)C(s).$$

In the next lemma we state a simple property that seems worthwhile to point out.

\begin{lemma} $A(s)$ has no zeroes or poles for Re$(s) \ge 0$.
\end{lemma}
\begin{proof} This follows from a cursory examination of the definition of $A(s)$; none of the zeroes or poles of the constituent functions occur to the right of 0.
\end{proof}

Let $R$ denote the subgroup of $\tR$ consisting of elements $h=(h_1,h_2)$ such
that $h_1\in G, h_2\in  \widetilde{F}$ and $\mu_2(h_1) =\mu_1(h_2)$. The above Theorem, along with our pullback formula, implies the following result.

\begin{theorem}\label{t:secondintegral}We have $$\int_{g \in Z(\A) R(\Q) \bs R(\A)} E_{\Upsilon}( \i(g_1,g_2) ,s)\overline{\Phi}(g_1)\Psi(g_2) \Lambda^{-1}(\det g_2)dg = A(s)L(3s + \frac{1}{2}, F \times g)$$ where $g = (g_1, g_2)$.
\end{theorem}

This new integral representation has a great advantage over the previous one: the Eisenstein series $E_{\Upsilon}(g,s)$ is much simpler than $E_{\Psi,\Lambda}(g,s)$ (even though it lives on a higher rank group). This is because it is induced from a \emph{one-dimensional} representation of the Siegel parabolic. Thus, it is more suitable for applications, especially with regard to special value results.

\begin{corollary}\label{coranal}$L(s, F \times g)$ can be continued to a meromorphic function on the entire complex plane. It's only possible pole to the right of the critical line Re($s$)$=\frac{1}{2}$ is at $s=1$.
\end{corollary}
\begin{proof} The integral representation of Theorem~\ref{t:secondintegral} immediately proves the meromorphic continuation. Furthermore by~\cite{ich2}, we know that the only possible poles of the Eisenstein series $E_{\Upsilon}( g ,s)$ to the right of $s=0$ are at $s=\frac{1}{6}$ and $s=\frac{1}{2}$. However, as we remark in the proof of Proposition~\ref{eisensteinrestrictionmodular}, there is no pole at $s=\frac{1}{2}$. So the only possible pole of the Eisenstein series in that half plane is at $s=\frac{1}{6}$ which corresponds to a pole of the $L$-functions at $s=1$.
\end{proof}

\subsection{Eisenstein series on Hermitian domains}
Let $$\G^+(\R) = \{g \in \G(\R) : \mu_2(g)>0\}.$$ Define the groups $G^+(\R), \tH{}^+(\R), \tF^+(\R)$ similarly.

Also recall the definitions of the symmetric domains $\H_n$, $\Ht_n$ from the section on notations. We define the `standard embedding' of $\Ht_2 \times \Ht_1$ into $\Ht_3$ by $$(Z_1, Z_2) \mapsto \begin{pmatrix}Z_1&\\&Z_2\end{pmatrix}.$$ We use the same notation $(Z_1, Z_2)$ to denote an element of $\Ht_2 \times \Ht_1$ and its image in $\Ht_3$ under the above embedding. Note that this embedding restricts to an embedding of $\H_2 \times \H_1$ into $\H_3$.

We also define another embedding $u$ of $\Ht_2 \times \Ht_1$ into $\Ht_3$ by $$u(Z_1, Z_2) = (Z_1, -\overline{Z_2}).$$ Clearly this embedding also restricts to an embedding of $\H_2 \times \H_1$ into $\H_3$.

Furthermore, the following is true, as can be verified by an easy calculation:

Let $g_1 \in \G_1(\R), g_2 \in \tF_1(\R),$ such that $g_1 (i) = Z_1, g_2 (i) = Z_2.$ In the event that $(Z_1, Z_2) \in \H_2 \times \H_1$ we may even take $g_1 \in G_1(\R), g_2 \in SL_2(\R).$

Then $$u(Z_1, Z_2) = \i(g_1, g_2)i.$$

Now, let us interpret the Eisenstein series of the last section as a function on $\Ht_3$. Recall the definitions of the sections $\Upsilon_v(s) \in
\text{Ind}_{P_{\tH}(\Q_v)}^{\tH(\Q_v)} (\Lambda_v \| \cdot \|_v^{3s}) $. Also, for $Z \in \Ht_n$, we set $\widehat{Z}= \frac{i}{2}(\overline{Z}^t -Z).$
\begin{lemma}\label{explicitupsilonformula}Let $g_\infty \in \tH{}^+(\R)$. Then
$$\Upsilon_\infty(g_\infty, s) = \det(g_\infty)^{\ell/2}\det(J(g_\infty, i))^{-\ell}\det(\widehat{g_\infty(i)})^{3(s+1/2)-\ell/2}$$
\end{lemma}
\begin{proof}
Let us write $g_\infty = m(A,v)n k_\infty$ where $m(A,v) \in M(\A)$, $n \in N(\A)$ and $k \in K_\infty^{\tH}$.
Then,~\eqref{e:upsilondefformula} and~\eqref{e:upsiooninf} tells us that  $$\Upsilon_\infty(g_\infty, s) = v^{-9(s+1/2)}|\det A|^{6s+3}\det(k_\infty)^{\ell/2}\det(J(k_{\infty}, i))^{-\ell}.$$

On the other hand, we can verify that $$\widehat{g_\infty(i)} = v^{-1}A\overline{A}^t$$ and therefore $$\det(\widehat{g_\infty(i)}) = v^{-3}|\det A|^2.$$ Also we see that $$J(g_\infty, i) = v(\overline{A}^t)^{-1}J(k_\infty, i)$$ which implies $$\det(J(g_\infty, i)) = v^3\det(\overline{A})^{-1}\det(J(k_\infty, i)).$$ Finally $$\det (g_\infty) = v^3 \det(k_\infty)\det(A)\det(\overline{A})^{-1}.$$ Putting the above equations together, we get the statement of the lemma.
\end{proof}
\begin{corollary}\label{eisensteinwelldefinedZ}Let $s \in \C$, $u_f \in \tH(\A_f)$ be fixed. Then the function $\Sigma$ on $\tH{}^+(\R)$ defined by $$\Sigma(g_\infty)=\det(g_\infty)^{-\ell/2}\det(J(g_\infty,i))^\ell
E_{\Upsilon}(u_fg_\infty, \frac{s}{3} + \frac{\ell}{6} -1/2)$$ depends only on $g_\infty(i)$.
\end{corollary}
\begin{proof} We have $$E_{\Upsilon}(u_fg_\infty, s) = \sum_{\gamma \in P_{\tH(\Q)} \bs \tH(\Q)} \Upsilon_\infty(\gamma_\infty g_\infty,s) \Upsilon_{f}(\gamma_fu_f,s).$$ So, by the above lemma, \begin{equation}\label{eisensteinwelldefinedexplicit}\Sigma(g_\infty)= \sum_{\gamma \in P_{\tH(\Q)} \bs \tH(\Q)}\det(\gamma)^{\ell/2}\det(J(\gamma,Z))^{-\ell}
\det(\widehat{\gamma(Z)})^{s}\Upsilon_{f}(\gamma_fu_f,s)\end{equation}
where $Z= g_\infty(i)$.

\end{proof}

Now,consider the coset decomposition
\begin{equation}\label{strongapproxgl2}\tF(\A) = \bigsqcup_{i=1}^h\tF(\Q)\tF^+(\R)\begin{pmatrix}t_i&\\&t_i^\ast\end{pmatrix}
U^{\tF}\end{equation}
where $t_i \in \tF(\A_f)$, $t_i^\ast = \overline{t_i}^{-1}$, and \begin{equation}\label{defuftilde}U^{\tF}=\prod_{p\notin S}K_p^{\tF}\prod_{p\in S_3}\Gamma_{0,p}^{\tF}\prod_{p\in S_1 \sqcup S_2}\Gamma_{0,p}'^{\tF}.\end{equation}

We note here that the constant $h$ comes up because the class number of $L$ may not be 1 and because the $\det$ map from $\Gamma_{0,p}'^\tF$ to $\Z_{L,p}^\times$ is not surjective. In particular, note that if $M=1$, we have $h=h(-d)$, the class number of $L$.

Also, we note that by the Cebotarev density theorem, we may choose $t_i$ such that $(N_{L/\Q}t_i) = q_i^{-1}$ where $q_i$ corresponds to an ideal of $\Z$ that splits in $L$. In particular $\gcd(q_i, MN)=1.$

Now, let \begin{align*}\Gamma_i &= SL_2(\Z)\cap \begin{pmatrix}t_i&\\&t_i^\ast\end{pmatrix}U^{\tF}\begin{pmatrix}t_i^{-1}&\\&(t_i^\ast)^{-1}
\end{pmatrix} \tF(\R)\\&=\Gamma_0(M) \cap \Gamma_0(Nq_i).\end{align*}

Also, we define the congruence subgroup $\Gamma_{M,N}$ of $Sp_4(\Z)$ by $$\Gamma_{M,N} = B(M)\cap U_2(N).$$

Recall the definition of $U^{\G}$ from~\eqref{def:kgtilde}. Let us define the compact open subgroup $U^G$ of $G(\A_f)$ by \begin{equation}\label{defUG}U^G = U^{\G} \cap G(\A).\end{equation} Observe that $$\Gamma_{M,N}= U^G Sp_4(\R) \cap Sp_4(\Q).$$

Next, put $$s_i = \begin{pmatrix}t_i&\\&t_i^\ast\end{pmatrix}$$ and $$r_i = \i(1, s_i) \in \tH_1(\A_f).$$For $Z\in \Ht_3$, define the Eisenstein series
$E_\Upsilon^i(Z;s)$ by  \begin{equation}\label{edefeisensteinhol}E_\Upsilon^i(Z;s)= \det(g_\infty)^{-\ell/2}\det(J(g_\infty, i))^\ell E_{\Upsilon}(r_ig_\infty, s/3 + \ell/6 -1/2),\end{equation} where $g_\infty \in \tH{}^+(\R)$ is such that $g_\infty(i) =Z$. We note that $E_\Upsilon^i(Z,s)$ is well defined by Corollary~\ref{eisensteinwelldefinedZ}.

Now, consider the function $E_\Upsilon^i(Z_1,Z_2;0)$ for $Z_1 \in \H_2, \Z_2 \in \H_1$.

\begin{proposition}\label{eisensteinrestrictionmodular}Assume $\ell \ge 6$. Then $E_\Upsilon^i(Z_1,Z_2;0)$ is a modular form of weight $\ell$ for $\Gamma_{M,N} \times \Gamma_i$. Furthermore, for any $s_0$, the function $E_\Upsilon^i(Z_1,Z_2;s_0)$ (which is not holomorphic in $Z_1, Z_2$ unless $s_0 =0$) transforms like $E_\Upsilon^i(Z_1,Z_2;0)$ under the action of $\Gamma_{M,N} \times \Gamma_i$.
\end{proposition}
\begin{proof} We know that $E_{\Upsilon}(g, s)$ converges absolutely and uniformly for $s > \frac{1}{2}.$ So if $\ell >6$, it follows that $E_\Upsilon^i(Z;0)$ is holomorphic. Furthermore, the case $\ell =6$ corresponds to the point $s=\frac{1}{2}$ of $E_{\Upsilon}(g,s)$. From the general theory of Eisenstein series, we know that the residue of $E_{\Upsilon}(g,s)$ restricted to $K_\infty^{\tH}$ at $s=\frac{1}{2}$ must be a constant function. However, because $E_{\Upsilon}(g,s)$ is an eigenfunction of $K_\infty^{\tH}$ with non-trivial eigencharacter, this residue must be zero. Hence $E_\Upsilon^i(Z;0)$ is a holomorphic function of $Z$ even for $\ell=6$.

Let $A\in \Gamma_{M,N}, B\in \Gamma_i$. It suffices to show that $$E_\Upsilon^i(AZ_1,BZ_2;s_0) = \det(J(A, Z_1))^\ell\det(J(B, Z_2))^\ell E_\Upsilon^i(Z_1,Z_2;s_0).$$
For $g= \begin{pmatrix}a&b\\c&d \end{pmatrix}$ denote $\widetilde{g} = \begin{pmatrix}a&-b\\-c&d \end{pmatrix}.$ Let $g_1 \in G_1(\R), g_2 \in SL_2(\R)$ such that $g_1 i =Z_1, g_2 i= Z_2$. Put $s'= s_0/3+\ell/6-1/2$. We have \begin{align*}E_\Upsilon^i(AZ_1,BZ_2;s_0)
&=E_\Upsilon^i(u(AZ_1,\overline{-BZ_2});s_0)\\
&=E_\Upsilon^i(\i(Ag_1,\widetilde{B}\widetilde{g_2})i;s_0)\\
&=\det(J(\i(Ag_1,\widetilde{B}\widetilde{g_2}), i))^\ell E_\Upsilon(r_i\i(Ag_1,\widetilde{B}\widetilde{g_2}),s')\\
&=\det(J(\i(Ag_1,\widetilde{B}\widetilde{g_2}), i))^\ell  E_{\Upsilon}(\i(Ag_1,s_i\widetilde{B}\widetilde{g_2}),s')\end{align*}
Now, because $s_i^{-1}\widetilde{B}s_i \in U^{\tF}$ we have
$$E_{\Upsilon}(\i(Ag_1,s_i\widetilde{B}\widetilde{g_2}),s')=
E_{\Upsilon}(\i(g_1,s_i\widetilde{g_2});s').$$ On the other hand, we can check that $$\det(J((Ag_1,\widetilde{B}\widetilde{g_2}), i))^\ell = \det(J(A, Z_1))^\ell\det(J(B, Z_2))^\ell\det(J(g_1, i))^\ell\det(J(g_2, i))^l.$$ Putting everything together, we see that
$$E_\Upsilon^i(AZ_1,BZ_2;s_0)=\det(J(A, Z_1))^\ell\det(J(B, Z_2))^\ell E_\Upsilon^i(Z_1,Z_2;s_0)$$ as required.
\end{proof}

\subsection{The integral representation in classical terms}

Henceforth, we assume $\ell \ge6$. Recall the definitions of the compact open subgroups $U^G, U^{\tF}$ from~\eqref{defUG}, \eqref{defuftilde} respectively. Let us define $U^R \subset R(\A_f)$ to be the subgroup consisting of elements $(g,h)$ with $g \in U^G$, $h\in U^{\tF}$ and $\mu_2(g) = \mu_1(h)$. Also put $K_\infty^R = K_\infty \times K_\infty^{\tF}$.
Note that $K^RK_\infty^R$ is a compact subgroup of $R(\A)$.

Also, define $V_{M,N} = [Sp_4(\Z) : \Gamma_{M,N}] [K^{\tF} : U^{\tF}],$ where $K^{\tF} = \prod_{p<\infty} K_p^{\tF}$. We now rephrase Theorem~\ref{t:secondintegral} in classical terms.

\begin{theorem}\label{t:classicalintegral} For any $k$, we have

\begin{align*}&\sum_i \Lambda^{-2} (t_i)\int_{ \Gamma_i \bs \H_1} \int_{\Gamma_{M,N} \bs \H_2} E_\Upsilon^i(Z_1,-\overline{Z_2};1-k) \overline{F(Z_1)} g(q_iZ_2) \det(Y_1)^\ell \det(Y_2)^\ell dZ_1dZ_2 \\&= V_{M,N}A(\frac{\ell-1-2k}{6} )  L(\frac{\ell}{2} -k, F\times g)\end{align*} where for $i= 1, 2$, we define the invariant measure $dZ_{i}$ on $\H_{3-i}$ by $$dZ_i = \frac{1}{2}(\det Y_i)^{i-4}dX_idY_i$$ where $Z_i = X_i +iY_i$.

\end{theorem}
\begin{proof} By Theorem~\ref{t:secondintegral}, it suffices to prove that for $g = (g_1, g_2)$, \begin{align}\label{integrallhs}&V_{M,N}\int_{ Z(\A) R(\Q) \bs R(\A)} E_{\Upsilon}( \i(g_1,g_2) ,\frac{\ell-1-2k}{6})\overline{\Phi}(g_1)\Psi(g_2) \Lambda^{-1}(\det g_2)dg\\ &=\sum_i \Lambda^{-2} (t_i)\int_{\mathfrak F_1 } \int_{\mathfrak F_2} E_\Upsilon^i(Z_1,-\overline{Z_2};1-k) \overline{F(Z_1)} g(q_iZ_2) \det(Y_1)^\ell \det(Y_2)^\ell dZ_1 dZ_2\end{align}

where $\mathfrak F_1$ is a fundamental domain for $\Gamma_i \bs \H_1$ and $\mathfrak F_2$ a fundamental domain for $\Gamma_{M,N} \bs \H_2.$
Now, the quantity inside the integral in~\eqref{integrallhs} is right invariant by $U^RK_\infty^R$. Also, we note that the volume of $U^RK_\infty^R$ is equal to $(V_{M,N})^{-1}$ (recall that we normalize the volume of the maximal compact subgroup to equal 1).

Hence we see that~\eqref{integrallhs} equals

\begin{equation} \int_{Z(\A) R(\Q) \bs R(\A)/U^RK_\infty^R} E_{\Upsilon}( \i(g_1,g_2) ,\frac{\ell-1-2k}{6})\overline{\Phi}(g_1)\Psi(g_2) \Lambda^{-1}(\det g_2)dg\end{equation}

Now, by strong approximation for $Sp_4(\A)$ and~\eqref{strongapproxgl2} we know that \begin{align*}&Z(\A) R(\Q) \bs R(\A)/U^RK_\infty^R\\&=\coprod_{i=1}^h\left(\Gamma_{M,N} \bs Sp_4(\R) / K_\infty \right) \times \begin{pmatrix}t_i&0\\0&t_i^\ast\end{pmatrix}\left(\Gamma_i\bs SL_2(\R)/SO(2) \right).\end{align*}

Suppose $g \in Sp_4(\R), h \in SL_2(\R)$. Also, put $s_i = \begin{pmatrix}t_i&\\&t_i^\ast\end{pmatrix}$, $r_i = \i(1, s_i)$, $g(i) = Z_1, h(i) = Z_2$.

We have  \begin{align*}E_{\Upsilon}( \i(g,s_ih) ,\frac{\ell-1-2k}{6}) &= E_{\Upsilon}( r_i \i(g,h),\frac{\ell-1-2k}{6})\\
&=\det(J(\i(g,h),i))^{-\ell}E_\Upsilon^i(Z_1,-\overline{Z_2};1-k)\end{align*}
 On the other hand $\overline{\Phi}(g) = \overline{F(Z_1)}\overline{\det(J(g, i))^{-\ell}}$
 and $\Psi(s_ih) = g(q_iZ_2)\det(J(h,i))^{-\ell}$.

The result now follows from the observations
$$\det(J(\i(g,h),i)) = \det(J(g,i)) \overline{\det(J(h,i))},$$
$$ |\det(J(g, i))|^2 = \det(Y_1), \ |\det(J(h, i))|^2 = \det(Y_2).$$
and the fact that the Haar measure $dg$ equals $dZ_1dZ_2$ under the above equivalence.
\end{proof}
Let us take a closer look at the quantity $A(\frac{\ell-1-2k}{6})$ that appears in the statement of the above theorem in the case when $k$ is an integer, $1\le k \le \frac{\ell}{2}-2$. Write $a \sim b$ if $a/b$ is rational. From the definition of $A(s)$, it is clear that $$A(\frac{\ell-1-2k}{6})\sim \frac{\pi^{4+k-2\ell}\overline{a(\Lambda)}\sqrt{d}}
{L(\ell+1-2k,\chi_{-d})\zeta(\ell-2k)\zeta(\ell+2-2k)} .$$
But it is well known that $\frac{L(\ell+1-2k,\chi_{-d})}{\pi^{\ell+1-2k} \sqrt{d}}$, $\frac{\zeta(\ell-2k)}{\pi^{\ell-2k}}$ and $\frac{\zeta(\ell+2-2k)}{\pi^{\ell+2-2k}}$ are all rational numbers. It follows that  \begin{equation}\label{al6half}A(\frac{\ell-1-2k}{6}) \sim \pi^{7k+1-5\ell}\overline{a(\Lambda)}.\end{equation}

\section{Near holomorphy, holomorphic projection and rationality properties}\label{s:algebraicity}
\subsection{Rationality of holomorphic Eisenstein series}

Suppose $f_1,f_2$ are modular forms of weight $\ell$ for some congruence subgroup $\Gamma$ of $Sp_{2n}(\Z)$ containing $\{ \pm 1\}$. We define the Petersson inner product $$\langle f_1, f_2 \rangle = \frac{1}{2} V(\Gamma)^{-1} \int_{\Gamma \bs \H_n}f_1(Z)\overline{f_2(Z)} (\det Y)^{\ell-n-1}dXdY$$ where $V(\Gamma) = [Sp_{2n}(\Z) : \Gamma]$.

Note that these definitions are independent of our choice for $\Gamma$.

 We henceforth use $E_{\Lambda, \ell}^i$ for $E_{\Upsilon}^i$ in order to show the dependence on $\Lambda, \ell$ and $a(F, \Lambda)$ for $a(\Lambda)$ to show the dependence on $F$. Moreover, for any other positive even integer $k$, we use $E_{\Lambda, k}^i(Z;s)$ to denote the Eisenstein series that is defined similarly except that the integer $\ell$ has been replaced by $k$ everywhere. In particular, we know that $E_{\Lambda, k}^i(Z;0)$ is a holomorphic Eisenstein series (of weight $k$), whenever $k \ge6$.

By a result of M. Harris~\cite[Lemma 3.3.5.3]{har97}, we know how Aut($\C)$ acts on the Fourier coefficients of $E_{\Lambda, k}^i(Z;0)$.
In particular he proves the following result.

\begin{proposition}[Harris]Let $k \ge 6$. The Fourier coefficients of $E_{\Lambda, k}^i(Z;0)$ lie in $\Q^{\text{ab}}$. Furthermore, if $\sigma\in \text{Gal}(\Q^{\text{ab}}/\Q)$, then $$E_{\Lambda, k}^i(Z;0)^\sigma =E_{\Lambda^\sigma, k}^i(Z;0)$$ where $E_{\Lambda, k}^i(Z;0)^\sigma$ is obtained by letting $\sigma$ act on the Fourier coefficients of $E_{\Lambda, k}^i(Z;0)$.
\end{proposition}

\subsection{Nearly holomorphic Eisenstein series}
We can write any $Z\in \Ht_n$ uniquely as $Z= X+iY$ where $X,Y$ are Hermitian and $Y$ is positive definite. We can also write any $Z\in \H_n$ uniquely as $Z= X+iY$ where $X,Y$ are symmetric and $Y$ is positive definite. These decompositions are compatible with each other in the obvious sense under the inclusion $\H_n \subset \Ht_n$.

We briefly recall Shimura's theory of differential operators and nearly holomorphic functions. A thorough exposition of this material can be found in his book~\cite{shibook2}.

Let $\H$ temporarily stand for $\H_n$ or $\Ht_n$. For a non negative integer $q$, we let $\mathcal{N}^q(\H)$ denote the space of all polynomials of degree $\le q$ in the entries of $Y^{-1}$ with holomorphic functions on $\H$ as coefficients.

Suppose $\Gamma$ is a congruence subgroup of $Sp_{2n}$ (if $\H= \H_n$) or $U(n,n)$ (if $\H= \Ht_n$). For a positive integer $k$, we let $\mathcal{N}_k^q(\H,\Gamma)$ stand for the space of functions $f\in \mathcal{N}^q(\H)$ satisfying $$f(\gamma Z) = \det(J(\gamma, Z))^k f(Z)$$ for all $\gamma \in \Gamma, Z\in \H$, with the standard additional (holomorphy at cusps) condition on the Fourier expansion if $\H= \H_1= \Ht_1$. It is well-known that $\mathcal{N}_k^q(\H,\Gamma)$ is finite dimensional. In particular, if $q=0$, then $\mathcal{N}_k^q(\H,\Gamma)$ is simply the corresponding space of weight $k$ modular forms.

We let $N=n^2$ if $\H=\Ht_n$ and $N=(n^2 +n)/2$ if $\H=\H_n$.

Whenever convergent, the Petersson inner product for nearly holomorphic forms is defined exactly as in the previous section.

Any $f \in \mathcal{N}_q^t(\H,\Gamma) $ has a Fourier expansion~\cite[p. 117]{shibook2} as follows:

$$f(Z) = \sum_{T\in \mathcal{L}}Q_T((2\pi Y)^{-1})e^{2 \pi i Tr TZ}$$ where $\mathcal{L}$ is a suitable lattice and for each $T$, $Q_T$ is a polynomial in $N$ variables and of degree $\le t$. For an automorphism $\sigma$ of $\C$ we define $$f^\sigma(Z) = \sum_{T\in \mathcal{L}}Q_T^\sigma((2\pi Y^{[\sigma]})^{-1})e^{2 \pi i Tr TZ}$$ where $Q_T^\sigma$ is obtained by letting $\sigma$ act on the coefficients of $Q_T$ and $$Y^{[\sigma]}= \begin{cases}Y^t & \text{ if } \H = \Ht_n \text{ and } \sqrt{-d}^\sigma= -\sqrt{-d}\\
Y & \text { otherwise } \end{cases}$$

We say that $f \in \mathcal{N}_q^t(\H,\Gamma;\overline{\Q}) $ if $f \in \mathcal{N}_q^t(\H,\Gamma) $ and $f^\sigma = f$ for all $\sigma \in \mathrm{Aut}(\C/\overline{\Q}).$  We will occasionally omit the weight $q$ and the congruence subgroup $\Gamma$ when we do not wish to specify those. In particular, we write $\mathcal{N}_q^{t}(\H;\overline{\Q})$ to denote $\bigcup_{\Gamma} \mathcal{N}_q^t(\H,\Gamma;\overline{\Q})$ where the union is taken over all congruence subgroups $\Gamma$.

Now, from~\eqref{eisensteinwelldefinedexplicit}, it is easy to see that for a positive integer $k$ (assume  $k \le \frac{\ell}{2}-2$ to ensure convergence) we have $E_{\Lambda, \ell}^i(Z;1-k) \in \mathcal{N}^{3(k-1)}(\Ht_3).$ Then, exactly the same proof as Proposition~\ref{eisensteinrestrictionmodular} tells us that the restriction of this function to $\H_2 \times \H_1$ is a nearly holomorphic modular form with respect to the appropriate subgroups. More precisely, we have \begin{equation}\label{restrictiontensor}E_{\Lambda, \ell}^i(Z_1, Z_2;1-k) \in  \mathcal{N}_\ell^{2(k-1)}(\H_2, \Gamma_{M,N}) \otimes \mathcal{N}_\ell^{(k-1)}(\H_1, \Gamma_i).\end{equation}

We remark here that for a general $f \in \mathcal{N}^{3(k-1)}(\Ht_3)$ we can only say that $f(Z_1, Z_2) \in \sum \mathcal{N}^{\lambda_1}(\H_2) \otimes \mathcal{N}^{\lambda_2}(\H_1)$ where the sum should be extended over all $(\lambda_1, \lambda_2)$ with $\lambda_1+\lambda_2 = 3(k-1)$. However, in this case, we know by~\eqref{eisensteinwelldefinedexplicit} the exact nature of the polynomial of degree $3(k-1)$; thus we can conclude that $\lambda_1 = 2(k-1), \lambda_2 = k-1$.

To prove the desired algebraicity result for critical $L$-values, we will need to know rationality properties for the nearly holomorphic modular forms in~\eqref{restrictiontensor}. That is the substance of the next proposition.
\begin{proposition}\label{nearholautc}Let $\ell\ge6$ and let $k$ be an integer satisfying $1\le k \le \frac{\ell}{2}-2$. Then the function $E_{\Lambda, \ell}^i(Z_1, Z_2;1-k)$ on $\H_2 \times \H_1$ belongs to $$\pi^{3(k-1)}\left(\mathcal{N}_\ell^{2(k-1)}(\H_2, \Gamma_{M,N};\overline{\Q}) \otimes \mathcal{N}_\ell^{(k-1)}(\H_1, \Gamma_i;\overline{\Q})\right).$$ Furthermore, for an automorphism $\sigma$ of $\C$, we have $$(\pi^{-3(k-1)}E_{\Lambda, \ell}^i(Z_1, Z_2; 1-k))^\sigma = \pi^{-3(k-1)}E_{\Lambda^\sigma, \ell}^i(Z_1, Z_2;1-k).$$\end{proposition}
\begin{proof} Since we already know~\eqref{restrictiontensor} and since the Fourier coefficients of $E_{\Lambda, \ell}^i(Z_1, Z_2; 1-k)$ are just sums of those of $E_{\Lambda, \ell}^i(Z;1-k)$, it is enough to prove that \begin{equation}\label{whatweneeds}(\pi^{-3(k-1)}E_{\Lambda, \ell}^i(Z;1-k))^\sigma = \pi^{-3(k-1)}E_{\Lambda^\sigma, \ell}^i(Z;1-k).\end{equation}
 For positive integers $p,q$, we have the (modified) Maass-Shimura differential operator $\Delta_q^p$ that acts on the space of nearly holomorphic forms of weight $q$ on $\Ht_3$. This operator is defined in~\cite[p. 146]{shibook2}. By~\cite[Theorem 14.12]{shibook2}, we know that $$\Delta_q^p\mathcal{N}_q^{t}(\Ht_3;\overline{\Q}) \subset \pi^{3p} \mathcal{N}_{q + 2p}^{t +3p}(\Ht_3;\overline{\Q}).$$ However, more is true; in fact \begin{equation}\label{autcdelta}(
(\pi i)^{-3p} \Delta_q^p f)^\sigma = (\pi i)^{-3p} \Delta_q^p (f^\sigma)\end{equation} whenever $f \in \mathcal{N}_q^{t}(\Ht_3).$ This easily follows from~\cite[p. 118]{shibook2} since the  Maass-Shimura operators are special cases of the operators considered there and the projection map is Aut($\C$)-equivariant. An alternative way to directly see~\eqref{autcdelta} is to observe that the action of the Maass-Shimura operator on the Fourier coefficients of a nearly holomorphic form can be explicitly computed and observed to satisfy the desired property. The details in the symplectic case were worked out by Panchishkin~\cite[Theorem 3.7]{pan}; the calculations in the unitary case are very similar.

 We know that $E_{\Lambda, \ell +2 -2k}^i (Z;0) \in \mathcal{N}_{\ell +2 - 2k}^{0}(\Ht_3;\overline{\Q})$. So, we can apply~\eqref{autcdelta} when $t=0, p=k-1, q=\ell+2-2k, f= E_{\Lambda, \ell +2 -2k}^i (Z;0).$

 Moreover, by the result of Harris stated in the previous section, $$E_{\Lambda, \ell +2 -2k}^i (Z;0)^\sigma = E_{\Lambda^\sigma, \ell +2 -2k}^i (Z;0).$$ So,~\eqref{whatweneeds} will follow if we know that
\begin{equation}\label{differentialeisen}\Delta_{\ell - 2(k-1)}^{k-1}E_{\Lambda, \ell +2 -2k}^i (Z;0) = c \cdot i^{3(k-1)} \cdot E_{\Lambda, \ell}^i(Z;1-k)\end{equation} for some rational number $c$ (The superscript $i$ should not be confused with the quantity $i =\sqrt{-1}$ that appears above!).

But~\eqref{differentialeisen} is precisely the content of Shimura's calculations in~\cite[(17.27)]{shibook2}. We remark here that the Eisenstein series Shimura considers has different sections than ours at the finite places dividing $MN$; however that does not make a difference because the differential operator only depends on the archimedean section. In particular, we apply~\cite[Theorem 12.13]{shibook2} to each term of the definition of our Eisenstein series using~\eqref{eisensteinwelldefinedexplicit} and observe that~\eqref{differentialeisen} follows with $c= 2^{-3(k-1)} c^{k-1}_{\ell - 2(k-1)}(\frac{\ell}{2}-k+1)$ where $c_q^p(s)$ is defined as in~\cite[(17.20)]{shibook2}.

\end{proof}

\subsection{Holomorphic projection}
Shimura observed~\cite[p. 123]{shibook2} that for $q>n+t$, there exists a holomorphic projection operator $\mathfrak{A}$ on $\mathcal{N}_q^{t}(\H_n)$. For a nearly holomorphic form $f \in \mathcal{N}_q^{t}(\H_n)$, $\mathfrak{A}f$ is a modular form of weight $q$ (i.e. an element of $\mathcal{N}_q^{0}(\H_n)$). For any cusp form $g$ of weight $q$ on $\H_n$, $$<f,g> = <\mathfrak{A}f,g>.$$

More precisely, by the proof of~\cite[Theorem 15.3]{shibook2}, we can write $$f = \mathfrak{A}f + L_q f'$$ where $L_q$ is a rational polynomial of certain differential operators and $f'$ is a certain nearly holomorphic form. The differential operators which are used to define $L_q$ are Aut$(\C)$-equivariant by~\cite[Theorem 14.12]{shibook2}. Thus, for an automorphism $\sigma$ of $\C$, we have $$f^\sigma =  (\mathfrak{A}f)^\sigma + L_q (f'^\sigma).$$ So we can conclude that $$\mathfrak{A}(f^\sigma) = (\mathfrak{A}f)^\sigma.$$

Furthermore because the space of modular forms is a direct sum of the space of Eisenstein series and the space of cusp forms, there exists an orthogonal projection from the space of modular forms on $\H_n$ to the space of cusp forms on $\H_n$. Because the space of Eisenstein series is preserved under automorphisms of $\C$, this cuspidal projection is also Aut$(\C)$-equivariant.

From the above comments we conclude the existence of a projection map $\mathfrak{A}_{\text{cusp}}$ from $\mathcal{N}_q^{t_1}
(\H_2, \Gamma_2) \otimes \mathcal{N}_q^{t_2}(\H_1, \Gamma_1)$ to $S_q
(\H_2, \Gamma_2) \otimes S_q(\H_1, \Gamma_1)$ for $q>2+t_i$ and congruence subgroups $\Gamma_2 \subset Sp_4, \Gamma_1 \subset SL_2$. This projection map satisfies, for any $\mathfrak{E}(Z_1, Z_2) \in \mathcal{N}_q^{t_1}
(\H_2, \Gamma_2) \otimes \mathcal{N}_q^{t_2}(\H_1, \Gamma_1$), $F^{(1)} \in S_q
(\H_2, \Gamma_2)$, $g^{(1)}\in S_q(\H_1, \Gamma_1)$, the following properties:
\begin{enumerate}
\item $\langle \langle\mathfrak{A}_{\text{cusp}}
    \mathfrak{E}(Z_1,Z_2),F^{(1)}(Z_1)\rangle, g^{(1)}(Z_2)\rangle = \langle \langle(\mathfrak{E}(Z_1,Z_2),F^{(1)}(Z_1)\rangle, g^{(1)}(Z_2)\rangle$,
\item ($\mathfrak{A}_{\text{cusp}}
    \mathfrak{E})^\sigma = \mathfrak{A}_{\text{cusp}}
    (\mathfrak{E}^\sigma)$.

\end{enumerate}
In particular, everything above can be applied to the case when $\mathfrak{E}(Z_1, Z_2) = \pi^{-3(k-1)}E^i_{\Lambda, \ell}(Z_1, Z_2; 1-k)$.

We use $g_i(z)$ to denote the cusp form $g(q_iz)$ on $\Gamma_0(Nq_i)$.
We can rewrite Theorem~\ref{t:classicalintegral} as follows.

\begin{align*}&\sum_i \Lambda^{-2} (t_i)\langle \langle E_{\Lambda, \ell}^i(Z_1,Z_2;1-k), F(Z_1)\rangle, g_i(Z_2)\rangle \\
&=\frac{V_{M,N}}{V(\Gamma_i)V(\Gamma_{M,N})}A(\frac{\ell-1-2k}{6})  L(\frac{\ell}{2} -k, F\times g)\end{align*}

Note that we have used the fact that $g_i$ has real Fourier coefficients.
Together with~\eqref{al6half} the above equation implies that \begin{equation}\label{6.6.1}\sum_i \Lambda^{-2} (t_i)\langle \langle E_{\Lambda,\ell}^i(Z_1,Z_2;1-k), F(Z_1)\rangle, g_i(Z_2)\rangle \sim \pi^{7k+1-5\ell}\overline{a(F, \Lambda)}L(\frac{\ell}{2} -k, F\times g).\end{equation}
\section{Deligne's conjecture}\label{s:mainresult}
\subsection{Motives and periods} Let $L(s, \mathcal{M})$ be the $L$-function associated to a motive $\mathcal{M}$ over $\Q$. Suppose $\mathcal{M}$ has coefficients in an algebraic number field $E$; then $L(s, \mathcal{M})$ takes values in $E \otimes_\Q \C$.

Note that $E$ sits naturally inside $E \otimes_\Q \C$. Let $d$ be the rank of $\mathcal{M}$ and $d^{\pm}$ the dimensions of the $\pm$ eigenspace of the Betti realization of $\mathcal{M}$. Deligne defined the motivic periods $c^{\pm}(\mathcal{M})$ and conjectured that for all ``critical points" $m$, $$\frac{L(m,\mathcal{M})}{(2 \pi i)^{md^{\epsilon}}c^\epsilon(\mathcal{M})} \in E$$

where $\epsilon = (-1)^m$.

    Now, let $F$, $g$ have algebraic Fourier coefficients. Assuming the existence of motives $M_F, M_g$ attached to $F, g$ respectively, Yoshida computed the critical points for $M_F \otimes M_g$. He also computed the motivic periods $c^\pm(M_F \otimes M_g)$ under the assumption that Deligne's conjecture holds for the degree 5 $L$-function for $F$. We note here that Yoshida only deals with the full level case; however as the periods remain the same (up to a rational number) for higher level, his results remain applicable to our case.

    Yoshida's computations~\cite[Theorem 13]{yosh} show that Deligne's conjecture implies the following reciprocity law: \begin{equation}\left(\frac{L(m, F \times  g)}{\pi^{4m+3\ell -4}\langle F, F \rangle \langle  g,  g \rangle}\right)^\alpha = \frac{L(m, F^\alpha \times  g^\alpha)}{\pi^{4m+3\ell -4}\langle F^\alpha, F^\alpha \rangle \langle  g^\alpha,  g^\alpha \rangle}\end{equation} for all $2-\frac{\ell}{2}\le  m \le\frac{\ell}{2}-1$, $\alpha \in \text{Aut}(\C)$.

In the next subsection we prove the above statement for all the critical points $m$ to the \emph{right} of Re($s)=\frac{1}{2}$ \emph{except} for the point $1$. The proof for the critical values to the left of Re($s)=\frac{1}{2}$ would follow from the expected functional equation. The proof that $L(1, F\times g)$ behaves nicely under the action of Aut($\C)$ would probably require further work because we do not know that this quantity is even finite (see Corollary~\ref{coranal}). Thus, the problem of extending our result to the remaining critical values is closely related to questions of analyticity and the functional equation for the $L$-function. These questions are also of interest for other applications, such as transfer to $GL(4)$ and will be considered in a future paper.

We also note that the integral representation (Theorem~\ref{t:secondintegral}) is of interest for several other applications. Indeed, we hope that this integral representation will pave the way
to stability, hybrid subconvexity, non-vanishing, non-negativity and $p$-adic results for the $L$-function under consideration. We intend to deal with these questions elsewhere.

\subsection{The main result}

 \begin{theorem}\label{t:specialreciprocity}Let $\ell \ge 6$. Further, assume that $F$ has totally real algebraic Fourier coefficients and define $$A(F, g;k) = \frac{L(\frac{\ell}{2}-k, F \times  g)}{\pi^{5\ell-4k-4}\langle F, F \rangle \langle  g,  g \rangle}.$$ Then, for $k$ be an integer satisfying $1\le k \le \frac{\ell}{2}-2$, we have: \begin{enumerate} \item  $A(F,g;k) \in \overline{\Q},$
 \item For all $\alpha \in$ Aut$(\C)$, $A(F,g;k)^\alpha = A(F^\alpha, g^\alpha;k).$
\end{enumerate}
\end{theorem}

\begin{proof}
Let $U$ be the least common multiple of $M,N$ and all the $q_i$. Let $\Gamma_1$ be the principal congruence subgroup of $Sp_4(\Z)$ of level $U$ and $\Gamma_2$ the principal congruence subgroup of $SL_2(\Z)$ of level $U$.
For each $i$, we can write \begin{equation}\mathfrak{A}_{\text{cusp}}(\pi^{-3(k-1)}
E_{\Lambda, \ell}^i(Z_1,Z_2;1-k)) = \sum_rF_1^r(Z_1)f_1^r(Z_2)\end{equation} where $F_1^r$ (resp. $f_1^r$) is a cusp form for $\Gamma_1$ (resp. $\Gamma_2$); all of weight $\ell$. Then \begin{equation}\sum_r\langle f_1^r, g_i\rangle \langle F_1^r, F\rangle = \pi^{-3(k-1)}\langle \langle E_{\Lambda, \ell}^i(Z_1,Z_2;1-k), F(Z_1)\rangle, g_i(Z_2)\rangle .\end{equation} We also have \begin{equation}\sum_r\langle (f_1^r)^\alpha, g_i^\alpha\rangle \langle (F_1^r)^\alpha, F^\alpha\rangle = \pi^{-3(k-1)}\langle \langle E_{\Lambda^\alpha, \ell}^i(Z_1,Z_2;1-k), F^\alpha(Z_1), g_i^\alpha(Z_2)\rangle \end{equation} using Proposition~\ref{nearholautc} and the properties of holomorphic projection stated above.

By~\eqref{6.6.1} we know that \begin{equation}\label{afgsim}A(F,g;k) =W \cdot (\overline{a(F,\Lambda)})^{-1} \cdot \sum_i \Lambda^{-2}(t_i)\frac{\sum_r\langle f_1^r, g_i\rangle \langle F_1^r, F\rangle}{\langle F, F \rangle \langle  g,  g \rangle}\end{equation} for some rational number $W$.

Making $\alpha$ act on both sides of the above equation we get
\begin{equation}A(F,g;k)^\alpha = W \cdot (\overline{a(F^\alpha,\Lambda^\alpha)})^{-1} \cdot \sum_i (\Lambda^\alpha)^{-2}(t_i)\left(\frac{\sum_r\langle f_1^r, g_i\rangle \langle F_1^r, F\rangle}{\langle F, F \rangle \langle  g,  g \rangle}\right)^\alpha.\end{equation} We also note that $\langle g, g\rangle = \langle g_i, g_i\rangle. $

Now by a result of Garrett~\cite[p. 460]{gar2}, we know that for each $r$, $$\left(\frac{\langle f_1^r, g_i\rangle \langle F_1^r, F\rangle}{\langle F, F \rangle \langle  g,  g \rangle}\right)^\alpha= \left(\frac{\langle (f_1^r)^\alpha, g_i^\alpha\rangle \langle (F_1^r)^\alpha, F^\alpha\rangle}{\langle F^\alpha, F^\alpha \rangle \langle  g^\alpha,  g^\alpha \rangle}\right).$$ so we have \begin{equation}A(F,g;k)^\alpha = W \cdot (\overline{a(F^\alpha,\Lambda^\alpha;k)})^{-1} \cdot \sum_i (\Lambda^\alpha)^{-2}(t_i)\left(\frac{\sum_r\langle (f_1^r)^\alpha, g_i^\alpha\rangle \langle (F_1^r)^\alpha, F^\alpha\rangle}{\langle F^\alpha, F^\alpha \rangle \langle  g^\alpha,  g^\alpha \rangle}\right).\end{equation} Using~\eqref{afgsim} for $F^\alpha, g^\alpha, \Lambda^\alpha$, we conclude that $$A(F,g;k)^\alpha=A(F^\alpha,g^\alpha;k).$$

\end{proof}
\textbf{Remark}. The above result was already known in the completely unramified case ($M=1, N=1$) by the work of B\"ocherer and Heim~\cite{heimboch} who used a different method.

\bibliography{lfunction}

\def\cprime{$'$} \newcommand{\noopsort}[1]{}
\begin{thebibliography}{10}

\bibitem{heimboch}
Siegfried B{\"o}cherer and Bernhard~E. Heim.
\newblock Critical values of {$L$}-functions on {${\rm GSp}\sb 2\times{\rm
  GL}\sb 2$}.
\newblock {\em Math. Z.}, 254(3):485--503, 2006.

\bibitem{bump}
Daniel Bump.
\newblock {\em Automorphic forms and representations}, volume~55 of {\em
  Cambridge Studies in Advanced Mathematics}.
\newblock Cambridge University Press, Cambridge, 1997.

\bibitem{fur}
Masaaki Furusawa.
\newblock On {$L$}-functions for {${\rm GSp}(4)\times {\rm GL}(2)$} and their
  special values.
\newblock {\em J. Reine Angew. Math.}, 438:187--218, 1993.

\bibitem{gar2}
Paul~B. Garrett.
\newblock On the arithmetic of {S}iegel-{H}ilbert cuspforms: {P}etersson inner
  products and {F}ourier coefficients.
\newblock {\em Invent. Math.}, 107(3):453--481, 1992.

\bibitem{grokud}
Benedict~H. Gross and Stephen~S. Kudla.
\newblock Heights and the central critical values of triple product
  {$L$}-functions.
\newblock {\em Compositio Math.}, 81(2):143--209, 1992.

\bibitem{har97}
Michael Harris.
\newblock {$L$}-functions and periods of polarized regular motives.
\newblock {\em J. Reine Angew. Math.}, 483:75--161, 1997.

\bibitem{heim}
Bernhard~E. Heim.
\newblock Pullbacks of {E}isenstein series, {H}ecke-{J}acobi theory and
  automorphic {$L$}-functions.
\newblock In {\em Automorphic forms, automorphic representations, and
  arithmetic (Fort Worth, TX, 1996)}, volume~66 of {\em Proc. Sympos. Pure
  Math.}, pages 201--238. Amer. Math. Soc., Providence, RI, 1999.

\bibitem{ich2}
Atsushi Ichino.
\newblock A regularized {S}iegel-{W}eil formula for unitary groups.
\newblock {\em Math. Z.}, 247(2):241--277, 2004.

\bibitem{ich}
Atsushi Ichino.
\newblock On the {S}iegel-{W}eil formula for unitary groups.
\newblock {\em Math. Z.}, 255(4):721--729, 2007.

\bibitem{lapral}
Erez Lapid and Stephen Rallis.
\newblock Positivity of {$L(\frac 12,\pi)$} for symplectic representations.
\newblock {\em C. R. Math. Acad. Sci. Paris}, 334(2):101--104, 2002.

\bibitem{micram}
Philippe Michel and Dinakar Ramakrishnan.
\newblock Consequences of the {G}ross/{Z}agier formulae: Stability of average
  {$L$}-values, subconvexity, and non-vanishing mod p.
\newblock To appear in the \emph{Serge Lang memorial volume}, Springer (2009).

\bibitem{miybook}
Toshitsune Miyake.
\newblock {\em Modular forms}.
\newblock Springer-Verlag, Berlin, 1989.
\newblock Translated from the Japanese by Yoshitaka Maeda.

\bibitem{pan}
A.~A. Panchishkin.
\newblock The {M}aass-{S}himura differential operators and congruences between
  arithmetical {S}iegel modular forms.
\newblock {\em Mosc. Math. J.}, 5(4):883--918, 973--974, 2005.

\bibitem{pitsch2}
Ameya Pitale and Ralf Schmidt.
\newblock Bessel models for lowest weight representations of
  $\mathrm{GSp}(4,\mathbb{R})$.
\newblock {\em Int. Math. Res. Notices}, 2009(7):1159--1212.
\newblock \textbf{doi:10.1093/imrn/rnn156}.

\bibitem{pitsch}
Ameya Pitale and Ralf Schmidt.
\newblock Integral {R}epresentation for ${L}$-functions for $\mathrm{GSp_4}
  \times \mathrm{GL_2}$.
\newblock {\em J. Number Theory}, 129(10):1272 -- 1324, 2009.

\bibitem{lfshort}
Abhishek Saha.
\newblock {$L$}-functions for holomorphic forms on $\mathrm{GSp(4)}$ $\times$
  $\mathrm{GL(2)}$ and their special values.
\newblock {\em Int. Math. Res. Notices}, 2009(10):1773 -- 1837.
\newblock \textbf{doi:10.1093/imrn/rnp001}.

\bibitem{sch}
Ralf Schmidt.
\newblock Iwahori-spherical representations of {${\rm GSp}(4)$} and {S}iegel
  modular forms of degree 2 with square-free level.
\newblock {\em J. Math. Soc. Japan}, 57(1):259--293, 2005.

\bibitem{shibook1}
Goro Shimura.
\newblock {\em Euler products and {E}isenstein series}, volume~93 of {\em CBMS
  Regional Conference Series in Mathematics}.
\newblock Published for the Conference Board of the Mathematical Sciences,
  Washington, DC, 1997.

\bibitem{shibook2}
Goro Shimura.
\newblock {\em Arithmeticity in the theory of automorphic forms}, volume~82 of
  {\em Mathematical Surveys and Monographs}.
\newblock American Mathematical Society, Providence, RI, 2000.

\bibitem{weiss}
Rainer Weissauer.
\newblock Four dimensional {G}alois representations.
\newblock {\em Ast\'erisque}, (302):67--150, 2005.
\newblock Formes automorphes. II. Le cas du groupe ${\rm{G}}Sp(4)$.

\bibitem{yosh}
Hiroyuki Yoshida.
\newblock Motives and {S}iegel modular forms.
\newblock {\em Amer. J. Math.}, 123(6):1171--1197, 2001.

\end{thebibliography}

\end{document}